\documentclass[12pt,reqno]{amsart}
\usepackage{amsthm,amsfonts,amssymb,euscript}
\usepackage{graphicx}
\usepackage{comment}

\def\alphab{\underline{\alpha}}
\def\betab{\underline{\beta}}
\def\chib{\underline{\chi}}
\def\chibh{\hat{\underline{\chi}}}
\def\chih{\hat{\chi}}

\def\divergence{\text{div}\,}
\def\Divergence{\text{Div}\,}
\def\D{\mathcal{D}}
\def\etab{\underline{\eta}}
\def\Hb{\underline{H}}
\def\Lb{\underline{L}}

\def\tr{\text{tr}}
\def\omegab{\underline{\omega}}

\def\tensor{\widehat{\otimes}}
\def\ub{\underline{u}}
\def\O{\mathcal{O}}

\def\Lh{\widehat{\mathcal{L}}}

\def\doubleint{\int\!\!\!\!\!\int}

\def\OSzerop{\mathcal{O}_{0,p}}

\def\OSzerofour{\mathcal{O}_{0,4}}

\def\OSonep{\mathcal{O}_{1,p}}

\def\OHtwo{{}^{(H)}\!\mathcal{O}}
\def\OHbtwo{{}^{(\Hb)}\!\mathcal{O}}

\def\Ozeroinfinity{\mathcal{O}_{0,\infty}}
\def\Ozerofour{\mathcal{O}_{0,4}}
\def\Ozerotwo{\mathcal{O}_{0,2}}
\def\Oonetwo{\mathcal{O}_{1,2}}
\def\Oonefour{\mathcal{O}_{1,4}}
\def\Ozero{{\mathcal{O}^{(0)}}}

\def\Rinitial{{\mathcal{R}^{(0)}}}
\def\Oinitial{{\mathcal{O}^{(0)}}}

\def\R{\mathcal{R}}
\def\Rb{\underline{\mathcal{R}}}

\def\Rzero{\mathcal{R}_0}
\def\Rone{\mathcal{R}_1}
\def\Rzerob{\underline{\mathcal{R}}_0}
\def\Roneb{\underline{\mathcal{R}}_1}

\def\Izero{\mathcal{I}_0}

\def\Wstar{{}^*W}
\def\Rstar{{}^*\!R}
\def\JNstar{{}^*\!J^{(N)}}
\def\Jfour{J^{(4)}}

\def\Jfourstar{{}^*\!J^{(4)}}

\def\DNRstar{{}^*\!D_N R}

\def\Jstar{{}^*\!J}
\def\trchibt{\widetilde{\tr \chib}}

\newtheorem{theorem}{Theorem}[section]
\newtheorem{lemma}[theorem]{Lemma}
\newtheorem{proposition}[theorem]{Proposition}
\newtheorem{corollary}[theorem]{Corollary}

\newtheorem{remark}[theorem]{Remark}

\newtheorem*{theoremC}{Main Estimates}
\newtheorem*{theoremD}{Existence Theorem}
\numberwithin{equation}{section}

\begin{document}

\title[Gravitational Collapse]{Energy Estimates and Gravitational Collapse}
\author[Pin Yu]{Pin Yu}
\address{Mathematical Sciences Center\\ Tsinghua University\\ Beijing, China}
\email{pin@math.tsinghua.edu.cn}
\thanks{The author is deeply indebted to Professor \emph{Sergiu Klainerman} and \emph{Igor Rodnianski} for many fruitful discussions on the problem. The author also would like to acknowledge Professor \emph{Demetrios Christodoulou} and \emph{Shing-Tung Yau} for explaining the insights and continuous encouragements. This work was partly done when the author was visiting Harvard University. He would like to thank the Department of Mathematics at Harvard University for their hospitality.}

\begin{abstract}
In this note, we study energy estimates for Einstein vacuum equations in order to prove the formation of black holes along evolutions. The novelty of the paper is that, we completely avoid using rotation vector fields to establish the global existence theorem of the solution. More precisely, we use only canonical null directions as commutators to derive energy estimates at the level of one derivatives of null curvature components. We show that, thanks to the special cancelations coming from the null structure of non-linear interactions, desirable estimates on curvatures can be derived under the short pulse ansatz due to Klainerman and Rodnianski \cite{K-R-09} (which is originally discovered by Christodoulou \cite{Ch}).
\end{abstract}
\maketitle
\setcounter{tocdepth}{1}
\tableofcontents
\section{Introduction}\label{introduction}
\subsection{A Brief History}
Penrose singularity theorem states that if in addition to the dominant energy condition, the space-time has a trapped surface, then the space-time contains \emph{singularities}. The \emph{weak cosmic censorship} conjecture asserts that under reasonable physical assumptions, singularities should be hidden from an observer at infinity by the event horizon of a black hole. Thus, by combining these two claims, to predict the existence of black holes, it suffices to exhibit one trapped surface in a space-time,. In other words, although many supplementary conditions are required, we regard the existence of a trapped surface as the presence of a black hole.

A major challenge in general relativity is to understand how trapped surfaces actually form due to the focusing of gravitational waves. In a recent breakthrough \cite{Ch}, Demetrios Christodoulou gave an answer to this long standing problem. He discovered a mechanism which is responsible for the dynamical formation of trapped surfaces in vacuum space-times. In the monograph \cite{Ch}, in addition to the Minkowskian flat data on a incoming null hypersurface, Christodoulou identified an open set of initial data (this is the \emph{short pulse} ansatz) on a outgoing null hypersurfaces. Based on the techniques developed by himself and Klainerman in the proof of the global stability of the Minkowski space-times \cite{Ch-K}, he managed to understand the whole picture of how various geometric quantities interact along the evolution. Once the estimates on curvatures are established in a large region of the space-time, the actual formation of trapped surfaces is easy to demonstrate. Christodoulou also proved a version of the same result for the short pulse data prescribed on past null infinity. This miraculous work provides the first global \emph{large data} result in general relativity (without symmetry assumptions) and opens the gate for many new developments on dynamical problems related to black holes.

In \cite{K-R-09}, Klainerman and Rodnianski extended aforementioned result of Christodoulou. They significantly simplified the proof of Christodoulou (from about six hundred pages to one hundred and twenty). They also enlarged the admissible set of initial conditions and show that the corresponding propagation estimates of connection coefficients and curvatures are much easier to derive. The relaxation of the propagation estimates are just enough to guarantee that a trapped surface still forms. Based on the trace estimates developed in a sequence of work of the authors towards the critical local well-posedness for Einstein vacuum equations, they reduced the number of derivatives needed of Christodoulou in the argument from two derivatives of the curvature (in Christodoulou's proof) to just one. More importantly, Klainerman and Rodnianski introduced a parabolic scaling in \cite{K-R-09} which is incorporated into Lebesgue norms and Sobolev norms. These new techniques allow them to capture the hidden \emph{smallness} of the nonlinear interactions among different small or large components of various geometric objects. The result of Klainerman and Rodnianski can be easily localized with respect to angular sectors, has the potential for further developments, see \cite{K-R-10}. We remark that Klainerman and Rodnianski only concentrated on the problem on a finite region. The question from past null infinity can be solved in a similar manner as in \cite{Ch} once one understand the picture on a finite region. The problem from past null infinity has been studied in a recent work by Reiterer and Trubowitz, \cite{R-T}.

\subsection{Novelty of the Paper}
One common feature of the proofs in \cite{Ch} and \cite{K-R-09} is that, in order to derive energy estimates on one or higher derivatives of curvature components, they all constructed three \emph{angular momentum vector fields} $O^{(1)}$, $O^{(2)}$ and $O^{(3)}$ which essentially captured almost rotational symmetry of the space-time. As far as the author aware, the main reason of using $O^{(i)}$'s is that the energy estimates on $(\Lh_{O^{(i)}} R)_{\alpha\beta\gamma\delta}$ behave well because one can take advantage of cancelations from the pseudo-symmetry of $O^{(i)}$'s. Here, the $\Lh_{O^{(i)}}$ is the \emph{modified Lie derivative} defined in \cite{Ch-K} and $R_{\alpha\beta\gamma\delta}$ is the curvature of the space-time. We can observe this advantage in the proof of stability of Minknowski space-time \cite{Ch-K} where $(\Lh_{O^{(i)}} R)_{\alpha\beta\gamma\delta}$ yields better decay estimates.

There is one obvious defect of the modified Lie derivative. As usual Lie derivatives, the $\Lh_{O^{(i)}}$ is \emph{not} tensorial in $O^{(i)}$. In fact, it evolves one derivative of $O^{(i)}$. In other words, if we use modified Lie derivatives, we may lose immediately one derivative. We have two remedies to this loss of derivatives: in \cite{Ch}, one relies on higher order derivative estimates; in \cite{K-R-09}, one makes use of more subtle trace estimates.

The above discussion can be summarized as follows: roughly speaking,  a good estimate on Lie derivatives relies on the almost symmetries $O^{(i)}$'s, but Lie derivatives causes a loss of derivative as we just explained. Hence, to avoid this loss, we shall give up the use of the pseudo-symmetries.

In this paper, we propose an approach to derive energy estimates on curvatures without constructing rotational vector fields at all.  In particular, in stead of using modified Lie derivatives, we work with covariant derivatives to save one derivative. Of course,  we have to pay a a price of controlling much more error terms. Nevertheless, this significantly simplifies the proofs compared to either \cite{Ch} or \cite{K-R-09}. In particular, compared to \cite{Ch}, we use only one derivative in curvature; compared to \cite{K-R-09}, instead of using trace inequality, all the estimates are derived from the classical Sobolev inequalities.

We also want to mention that in the thesis of L. Bieri, see \cite{B-Z}, based on a more general asymptotic assumptions, she gave a simplified proof of the stability of Minkowski space-time. She managed to derive decay from the time vector field and the conformal scaling of the space-time which allowed her to circumvent rotational vector fields in that situation. But the current situation is different from \cite{B-Z}: we do not use Lie derivatives at all and the lower regularity forces us to explore more structures from the Einstein equations.

\subsection{Structure of the Proof}
The main observation arises from the second Bianchi identities. Roughly speaking, they explicitly show how one expresses angular or rotational derivatives of curvature components $\nabla \Psi$ in terms of some null derivatives of curvature components plus lower order non-linear terms. Namely, they can be written schematically as
\begin{equation*}
\nabla_L \Psi = \nabla \Psi + l.o.t.,
\end{equation*}
or
\begin{equation*}
\nabla_{\Lb} \Psi = \nabla \Psi + l.o.t.,
\end{equation*}
where we use lower order terms $l.o.t.$ to collect all nonlinear interactions and $L$, $\Lb$ are two standard null directions under the framework of double null foliations. Thus, up to lower order corrections, to obtain the estimates on $\nabla \Psi$ rotational derivatives on curvatures, it suffices to control $\nabla_L \Psi$ or $\nabla_{\Lb} \Psi$ null derivatives of curvatures. Thus, we identify our main targets to be $D_L R_{\alpha\beta\gamma\delta}$ and $D_{\Lb} R_{\alpha\beta\gamma\delta}$ and we shall derive energy estimates for them.

According the above idea, we use the modified short pulse ansatz proposed Klainerman and Rodnianski in \cite{K-R-09}. We remark that this ansatz allows more large components than the original short pulse data discovered by Christodoulou in \cite{Ch}. We shall derive the energy estimates based on Bel-Robinson tensors associated to $D_L R_{\alpha \beta \gamma \delta}$ and $D_{\Lb} R_{\alpha \beta \gamma \delta}$. To deal with error terms, namely terms $I$, $J$ and $K$ in the proof in following sections, we have to take account of the special structure of those terms. In reality, some generic terms in error term may cause a loss of $\delta^{-\frac{1}{2}}$ which prevent us from closing the bootstrap argument. To avoid this loss, there are typically three techniques to use:

\begin{enumerate}
\item[1)] We bound a product of two term in $L^2_{(sc)}$ by two $L^4_{(sc)}$ estimates on each term instead of one $L^{\infty}_{(sc)}$ estimate on one of them and one $L^2_{(sc)}$ estimate on the other. In most of the situation, this trick saves a $\delta^{\frac{1}{4}}$.

\item[2)] When we integrate a product of terms on some null hypersurface or on a domain in the space-time, we use integration by parts to move a bad derivative, typically $\nabla_3$ or $\nabla_4$, from one term (for whom this bad derivative may cause a loss of $\delta^{-\frac{1}{2}}$) to another (for whom this bad derivative is not bad at all, namely, there is no loss in $\delta$). Combined with Bianchi identities, this procedure may save a $\delta^{\frac{1}{2}}$.

\item[3)] For some generic term in error estimates mentioned above, though it appears that it causes a loss of $\delta$ (which can not be retrieved by using the trick 1) and 2)), we can use in fact either signature considerations or a precise computations to show that this term does not show up at all in the error estimates. This manifests the special cancelations in the error terms.\footnote{\quad This cancelation can also be observed much more directly by another way of deriving energy estimates, namely, multiplying Bianchi identities and integrating directly on a give domain. The author would like to thank Igor Rodnianski for communicating this idea.}
\end{enumerate}

%



The whole proof is to combine these three tricks. The paper is organized as follows: in next section, we recall basic definitions and estimates from \cite{K-R-09} and we state the main theorem; in following sections, we derive energy estimates for derivatives of curvature components in the following order: $\nabla_4 \alpha$, $\nabla_3 \alphab$ and then $\nabla_4 \Psi_1$ for $\Psi \neq \alpha$ and $\nabla_3 \Psi$ for $\Psi \neq \alphab$.

\section{Main Result}

\subsection{The Double Null Foliation Framework} We briefly recall the double null foliation formalism,  See \cite{Ch} for more precise definitions. We use $\D = \D(1,\delta)$ to denote the underlying space-time and use $g$ to denote the background metric. We assume that $\D$ is spanned by a double null foliation generated by two optical functions $u$ and $\ub$ and they increase towards the future, $0 \leq u \leq 1$ and $0 \leq \ub \leq \delta$. We use $H_u$ / $\Hb_{\ub}$ to denote the outgoing / incoming null hypersurfaces generated by the level surfaces of $u$ / $\ub$. We use $S_{u,\ub}$ to denote the space-like two surface $H_u \cap \Hb_{\ub}$. We denote by $H_u^{(\ub_1,\ub_2)}$ the region of $H_u$ defined by $\ub_1 \leq \ub \leq\ub_2$; similarly, we can define $\Hb_{\ub}^{(u_1,u_2)}$.

\begin{minipage}[!t]{0.3\textwidth}
  \includegraphics[width=2.4in]{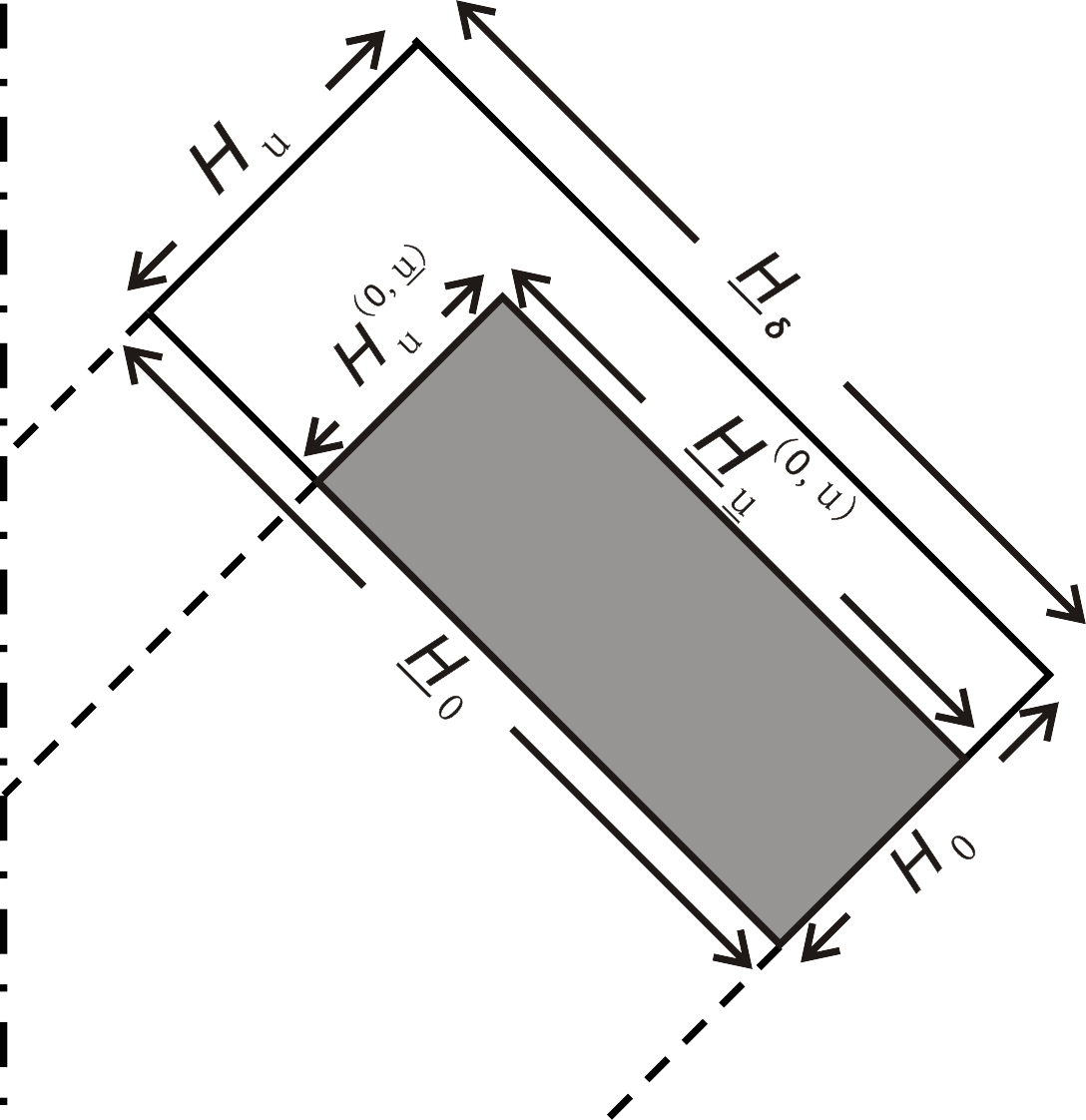}
\end{minipage}
\hspace{0.1\textwidth}
\begin{minipage}[!t]{0.5\textwidth}
The shaded region on the right represents the domain $\D(u,\delta)$ with $0 \leq u \leq 1$.  The function $u$ is in fact defined from $-1$ to $\delta$. When $u \leq 0$, this part of $H_0$ is assumed to be a flat light cone in Minkowski space with vertex located at $\ub = -1$. We use $r_0 \sim 2$ to measure the maximal radius of the flat part of $H_0$. In \cite{K-R-09}, the trapped surface forms at $\ub \sim 1$ and $u = \delta$.
\end{minipage}

Let $(L,\Lb)$ be the null geodesic generators of the double null foliation and we define the lapse function $\Omega$ by $ g(L,\Lb) = -\dfrac{2}{\Omega^2}$. The normalized null pair $(e_3, e_4)$ is defined by $e_3 = \Omega \Lb,~ e_4 = \Omega L, ~ g(e_3,e_4)= -2 $.
On sphere $S_{u,\ub}$ we choose an arbitrary orthonormal frame $(e_1,e_2)$. We call $(e_1, e_2, e_3, e_4)$ a \emph{null frame}.\footnote{\quad We use Greek letters $\alpha, \beta, \cdots$ to denote an index from $1$ to $4$ and Latin letters $a, b, \cdots$ to denote an index from $1$ to $2$. We also $N$ to denote the null direction either $L$ or $\Lb$. Repeated indices are always understood as taking sums}

We use $D$ to denote Levi-Civita connection defined by the metric $g$ and we define the \emph{connection coefficients} as follows,
\begin{align*}
 \chi_{ab}&= g(D_b e_4, e_b), ~ \eta_a = -\frac{1}{2}g(D_3 e_a, e_4), ~ \omega = -\frac{1}{4} g(D_4 e_3, e_4),\\
\chib_{ab}&=  g(D_b e_3, e_b), ~ \etab_a = -\frac{1}{2}g(D_4 e_a, e_3), ~ \omegab = -\frac{1}{4}g(D_3 e_4, e_3), ~\zeta_a  = \frac{1}{2} g(D_a e_4, e_3),
\end{align*}
where $D_\alpha = D_{e_\alpha}$. On $S_{u.\ub}$, $\nabla$ is the induced connection; $\nabla_3$ and $\nabla_4$ are the projections to $S_{u.\ub}$ of the covariant derivatives $D_3$ and $D_4$.

Given a Weyl field $W$, we introduce its null decomposition with respect to the given null frame,
\begin{align*}
 \alpha(W)_{ab} &= W(e_a, e_4, e_b, e_4), \quad  \beta(W)_a = \frac{1}{2}W(e_a,e_4,e_3,e_4), \quad \rho(W) = \frac{1}{4} W(e_4,e_3,e_4,e_3), \\
 \alphab(W)_{ab} &= W(e_a, e_3, e_b, e_3) \quad \betab(W)_a = \frac{1}{2}W(e_a,e_3,e_3,e_4), \quad \sigma(W) = \frac{1}{4}{} ^*W(e_4,e_3,e_4,e_3),
\end{align*}
where ${}^*W$ is the space-time Hodge dual of $W$. When $W$ is the Weyl curvature tensor, we use $\alpha, \alphab, \beta, \betab, \rho, \sigma$ to denote its null components.

We recall the null structure equations for the Einstein vacuum space-times (see \cite{K-R-09}). The originally Einstein field equations are
\begin{equation*}
R_{\alpha \beta} =0,
\end{equation*}
where $R_{\alpha\beta}$ is the Ricci curvature of the underlying space-times. We express this tensorial equations by using the null frame, by definition, this yields the null structure equations. We only list the transport type null structure equations which are relevant to the current work.
\begin{equation}\label{NSE_L_chi}
 \nabla_4 \tr \chi + \frac{1}{2}(\tr \chi)^2 = -|\chih|^2-2\omega \tr \chi, \quad \nabla_4 \chih + \tr \chi \,\chih = -2\omega \chih -\alpha,
\end{equation}
\begin{equation}\label{NSE_Lb_chib}
 \nabla_3 \tr \chib + \frac{1}{2}(\tr \chib)^2 = -|\chibh|^2-2\omegab \tr \chib, \quad \nabla_3 \chibh + \tr \chib \, \chibh = -2\omegab \chibh -\alphab,
\end{equation}
\begin{equation}\label{NSE_L_eta or Lb_etab}
 \nabla_4 \eta = -\chi \cdot (\eta -\etab)-\beta, \quad \nabla_3 \etab = -\chib \cdot (\etab -\eta)+\betab ,
\end{equation}
\begin{equation}\label{NSE_L_omegab}
 \nabla_4 \omegab = 2\omega \omegab +\frac{3}{4}|\eta-\etab|^2-\frac{1}{4}(\eta -\etab)\cdot(\eta +\etab)-\frac{1}{8}|\eta +\etab|^2 + \frac{1}{2}\rho,
\end{equation}
\begin{equation}\label{NSE_Lb_omega}
 \nabla_3 \omega = 2\omegab \omega +\frac{3}{4}|\eta-\etab|^2+\frac{1}{4}(\eta -\etab)\cdot(\eta +\etab)-\frac{1}{8}|\eta +\etab|^2 + \frac{1}{2}\rho,
\end{equation}
\begin{equation}\label{NSE_L_chib}
 \nabla_4 \tr \chib + \frac{1}{2}\tr \chi \, \tr \chib=2 \omega \tr \chib + 2 \divergence \etab  + 2|\etab|^2 +  2\rho - \chih \cdot \chibh,
\end{equation}
\begin{equation}\label{NSE_Lb_tr_chi}
 \nabla_3 \tr \chi + \frac{1}{2}\tr \chib \, \tr \chi=2 \omegab \tr \chi + 2 \divergence \eta + 2|\eta|^2 +  2\rho - \chih \cdot \chibh,\footnote{\quad $\divergence$ is the divergence on $S_{u, \ub}$ and $\Divergence$ is the space-time divergence.}
\end{equation}
\begin{equation}\label{NSE_L_chibh}
 \nabla_4 \chibh + \frac{1}{2} \tr \chi \chibh = \nabla \tensor \etab +2\omega \chibh -\frac{1}{2}\tr \chib \chih +\etab \tensor \etab,
\end{equation}
\begin{equation}\label{NSE_Lb_chih}
 \nabla_3 \chih + \frac{1}{2} \tr \chib \chih = \nabla \tensor \eta +2\omegab \chih -\frac{1}{2}\tr \chi \chibh +\eta \tensor \eta.
\end{equation}

We also express the Bianchi equations relative to the null frame to derive null Bianchi equations.\footnote{\quad We can eliminate $\zeta$ by $\zeta = \frac{1}{2}(\eta -\etab)$.}
\begin{equation}\label{NBE_Lb_alpha}
 \nabla_3 \alpha + \frac{1}{2}\tr \chib \alpha =\nabla \tensor \beta + 4\omegab \alpha - 3(\chih \rho + ^*\!\chih \sigma)+(\zeta + 4\eta)\tensor \beta,
\end{equation}
\begin{equation}\label{NBE_L_beta}
\nabla_4 \beta + 2 \tr \chi \beta = \divergence \alpha - 2\omega \beta +\eta \cdot \alpha ,
\end{equation}
\begin{equation}\label{NBE_Lb_beta}
\nabla_3 \beta +  \tr \chib \beta = \nabla \rho + ^*\! \nabla \sigma + 2\omegab \beta + 2\chih \cdot \betab + 3(\eta \rho + ^*\!\eta \sigma),
\end{equation}
\begin{equation}\label{NBE_L_sigma}
 \nabla_4 \sigma + \frac{3}{2} \tr \chi \sigma = -\divergence ^*\! \beta + \frac{1}{2}\chibh \cdot ^*\!\alpha - \zeta \cdot ^*\! \beta -2\etab \cdot ^*\!\beta,
\end{equation}
\begin{equation}\label{NBE_Lb_sigma}
 \nabla_3 \sigma + \frac{3}{2} \tr \chib \sigma = -\divergence ^*\! \betab + \frac{1}{2}\chih \cdot ^*\!\alphab - \zeta \cdot ^*\! \betab -2\eta \cdot ^*\!\betab,
\end{equation}
\begin{equation}\label{NBE_L_rho}
 \nabla_4 \rho + \frac{3}{2} \tr \chi \rho = \divergence \beta -\frac{1}{2}\chibh \cdot \alpha + \zeta \cdot \beta + 2\etab \cdot \beta,
\end{equation}
\begin{equation}\label{NBE_Lb_rho}
 \nabla_3 \rho + \frac{3}{2} \tr \chib \rho = -\divergence \betab-\frac{1}{2}\chih \cdot \alphab + \zeta \cdot \betab - 2\eta \cdot \betab,
\end{equation}
\begin{equation}\label{NBE_L_betab}
\nabla_4 \betab +  \tr \chi \betab = -\nabla \rho + ^*\! \nabla \sigma + 2\omega \betab + 2\chibh \cdot \beta - 3(\etab \rho - ^*\!\etab \sigma),
\end{equation}
\begin{equation}\label{NBE_Lb_betab}
\nabla_3 \betab + 2 \tr \chib \, \betab = -\divergence \alphab - 2\omegab \betab +\etab \cdot \alphab,
\end{equation}
\begin{equation}\label{NBE_L_alphab}
 \nabla_4 \alphab + \frac{1}{2}\tr \chi \alphab =-\nabla \tensor \betab + 4\omega \alphab - 3(\chibh \rho - ^*\!\chibh \sigma)+(\zeta - 4\etab)\tensor \betab,
\end{equation}

\subsection{Energy Estimates Scheme}\label{Energy Estimates Scheme}
We review our scheme for energy estimates on Weyl fields, see \cite{Ch-K} for the original resource.
Assume a Weyl field $W_{\alpha\beta\gamma\delta}$ and its Hodge dual $\Wstar_{\alpha\beta\gamma\delta}$ solve following divergence equations with source terms
\begin{equation}\label{divergence of W}
D^\alpha W_{\alpha\beta\gamma\delta} = J_{\beta\gamma\delta}, \quad D^\alpha \Wstar_{\alpha\beta\gamma\delta} = \Jstar_{\beta\gamma\delta},
\end{equation}
$J_{\alpha\beta\gamma}$ and $\Jstar_{\beta\gamma\delta}$ are called \emph{Weyl currents}.
\begin{remark}
For vacuum, the curvature tensor $R_{\alpha\beta\gamma\delta}$ is a Weyl field with zero currents
\begin{equation}\label{divergence of R}
D^\alpha R_{\alpha\beta\gamma\delta} = 0, \quad D^\alpha \Rstar_{\alpha\beta\gamma\delta} = 0.
\end{equation}
\end{remark}
The Bel-Robinson tensor $Q[W]_{\alpha\beta\gamma\delta}$ associated to $W_{\alpha\beta\gamma\delta}$ is defined as follows \footnote{\quad We shall use short hand notations $Q$ for $Q[W]$, $\alpha$ for $\alpha(W)$, $\beta$ for $\beta(W)$, $\rho$ for $\rho(W)$, ..., if there is no confusion in the context.}
\begin{equation*}
Q[W]_{\alpha\beta\gamma\delta}=W_{\alpha\mu\gamma\nu}W_{\beta}{}^{\mu}{}_{\delta}{}^{\nu}+\Wstar_{\alpha\mu\gamma\nu}{}\Wstar_{\beta}{}^{\mu}{}_{\delta}{}^{\nu}.
\end{equation*}
It is fully symmetric and traceless in all pair of indices. Moreover, it satisfies the dominant energy condition which allows one to recover estimates for Weyl field $W$. In pratical terms, this condition can be expressed by formulas,
\begin{align*}
Q_{4444} = 2 |\alpha|^2, \quad Q_{3333}= 2 |\alphab|^2, \quad Q_{4443}= 4 |\beta|^2, \quad  Q_{3334}= 4 |\betab|^2, \quad Q_{4433}= 4(\rho^2 +\sigma^2).
\end{align*}
We also list other null components of $Q$ for future use,
\begin{align}\label{components of Q}
Q_{a444} &= 4\alpha_{ab}\beta_b, ~ Q_{a333} = -4\alphab_{ab}\betab_b,\notag \\
Q_{a344} &= 4\rho\beta_a - 4\sigma \, ^*\!\beta_a, ~ Q_{a433} = -4\rho\betab_a - 4\sigma \, ^*\!\betab_a, \\
Q_{ab44} &= 2|\beta|^2 +2\rho \alpha -2\sigma\,^*\!\alpha,  ~ Q_{ab33} = 2|\betab|^2 +2\rho \alphab+2\sigma\,^*\!\alphab, \notag\\
Q_{ab34} &= -2\beta\tensor \betab + 2(\rho^2+\sigma^2)\delta_{ab}.\notag
\end{align}
In view of \eqref{divergence of W}, $Q$ enjoys the following divergence equations
\begin{equation}\label{divergence of Q}
D^\alpha  Q_{\alpha\beta\gamma\delta}= W_{\beta}{}^{\mu}{}_{\delta}{}^{\nu}J_{\mu\gamma\nu} + W_{\beta}{}^{\mu}{}_{\gamma}{}^{\nu}J_{\mu\delta\nu} + \Wstar_{\beta}{}^{\mu}{}_{\delta}{}^{\nu}{}\Jstar_{\mu\gamma\nu} + \Wstar_{\beta}{}^{\mu}{}_{\gamma}{}^{\nu}{}\Jstar_{\mu\delta\nu}.
\end{equation}
Given vector fields $X$, $Y$ and $Z$, we define the current $P[W](X,Y,Z)_{\alpha}$ associated to  $X$, $Y$, $Z$ and $W$ by
\begin{equation*}
 P_\alpha = P[W](X,Y,Z)_{\alpha} = Q_{\alpha\beta\gamma\delta}X^\beta Y^\gamma Z^\delta.
\end{equation*}
Thus,
\begin{equation}\label{divergence of P}
 D^\alpha P_{\alpha} = D^\alpha Q_{\alpha XYZ} + (\pi \cdot Q)(X,Y,Z),
\end{equation}
where $^{(X)}\pi$ is the deformation tensor of $X$ defined by $^{(X)}\pi_{\alpha\beta} = \frac{1}{2}(D_\alpha X_\beta + D_\beta X_\alpha)$ and
\begin{equation*}
(\pi \cdot Q)(X,Y,Z)=Q_{\alpha\beta\gamma\delta}{}^{(X)}\!\pi^{\alpha\beta}Y^{\gamma}Z^{\delta}+Q_{\alpha\beta\gamma\delta}{}^{(Y)}\!\pi^{\alpha\beta}Z^{\gamma}X^{\delta}+Q_{\alpha\beta\gamma\delta}{}^{(Z)}\!\pi^{\alpha\beta}X^{\gamma}Y^{\delta}.
\end{equation*}
We integrate \eqref{divergence of P} on $\D(u,\ub)$ to derive the fundamental energy identity \footnote{\quad $L$ and $\Lb$ are corresponding normals of the null hypersurfaces $H_u$ and $\Hb_{\ub}$.}

\begin{minipage}[!t]{0.2\textwidth}
  \includegraphics[width = 2 in]{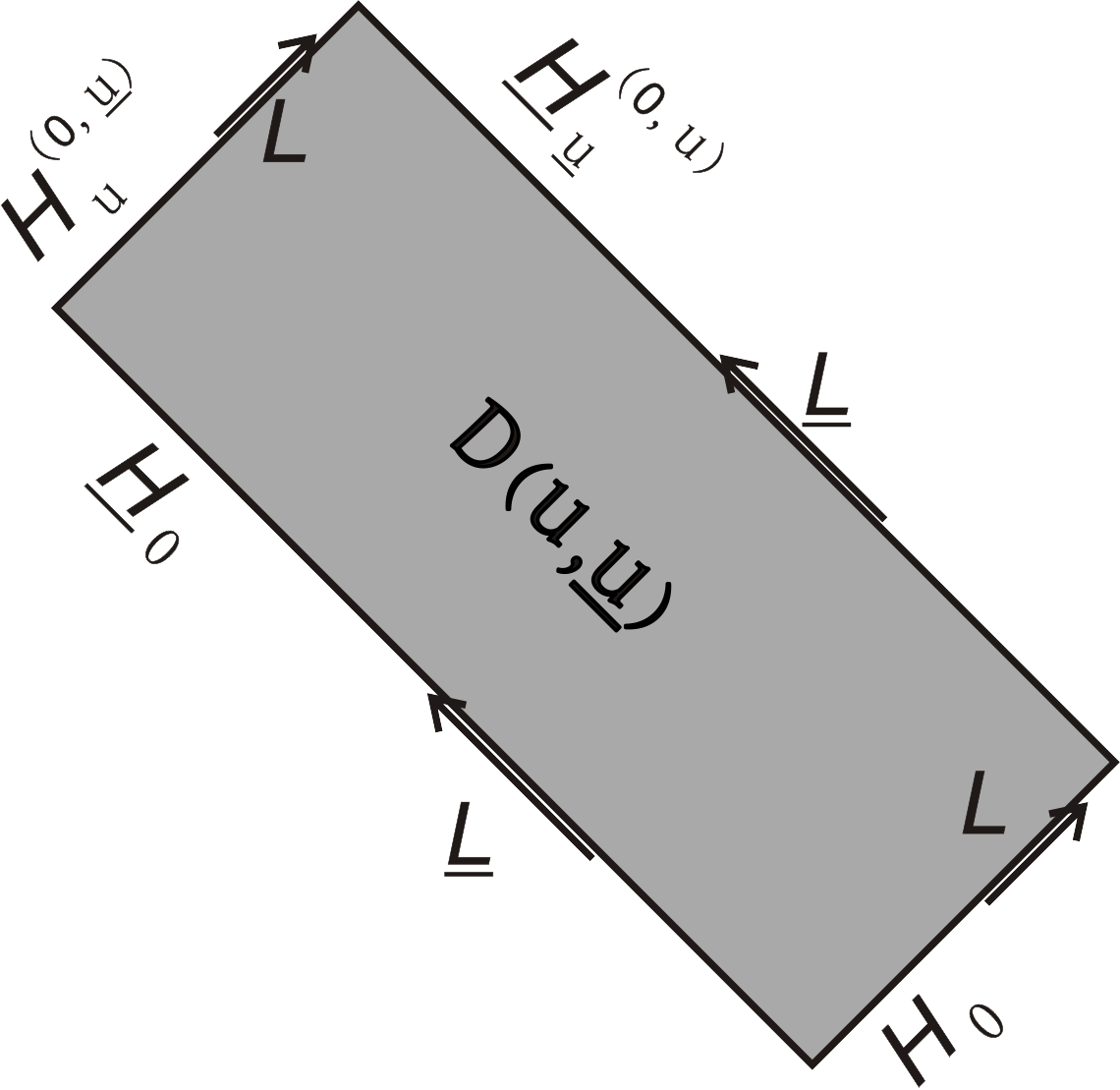}
\end{minipage}
\hspace{0.1\textwidth}
\begin{minipage}[!t]{0.6\textwidth}
\begin{align}\label{basic energy identity}
 &\quad\int_{H_u} Q(X,Y,Z,L)+\int_{\Hb_{\ub}} Q(X,Y,Z,\Lb) \notag\\
 &= \int_{H_0} Q(X,Y,Z,L)+\int_{\Hb_{0}} Q(X,Y,Z,\Lb)\\
&+\doubleint_{\D(u,\ub)} \Divergence Q (X,Y,Z) + \doubleint_{\D(u,\ub)} (\pi \cdot Q)(X,Y,Z).\notag
\end{align}
\end{minipage}

We list non-zero components of deformation tensors of $L$ and $\Lb$ as well as the non-zero component of $D^{\mu} \Lb^{\nu}$ and $D^{\mu} \Lb^{\nu}$ for future use\footnote{\quad We do not use $e_3$ and $e_4$ in order to avoid $\nabla_4 \omega$ and $\nabla_3 \omegab$ which do not have certain $L^4$ estimates, see \cite{K-R-09} for details.}:
\begin{align*}
 ^{(L)}\!\pi_{33} &= -8\Omega^{-1} \omegab, \quad ^{(L)}\!\pi_{3a} = 2\Omega^{-1}\eta^a, \quad  ^{(L)}\!\pi_{ab} = \Omega^{-1}\chi^{ab},\\
  ^{(\Lb)}\!\pi_{33} &= -8\Omega^{-1} \omega, \quad ^{(\Lb)}\!\pi_{4a} = 2\Omega^{-1}\etab^a, \quad  ^{(\Lb)}\!\pi_{ab} = \Omega^{-1}\chib^{ab},\\
 D^{4} L^{4} &= 2\omegab, \quad D^{4} L^{a} = D^{a} L^{4} = -\Omega^{-1}\eta^a, \quad D^{a} L^{b} = \Omega^{-1}\chi^{ab},\\
  D^{3} \Lb^{3} &= 2\omega, \quad D^{3} \Lb^{a} = D^{a} \Lb^{3} = -\Omega^{-1}\etab^a, \quad D^{a} \Lb^{b} = \Omega^{-1}\chib^{ab}.
\end{align*}
	
We consider one derivative of curvature $R_{\alpha\beta\gamma\delta}$ in a null direction $N$. One easy but important observations is that $D_N R_{\alpha\beta\gamma\delta}$ is still a Weyl field. We can commute $D_N$ with \eqref{divergence of R} to derive\footnote{\quad Recall that $D_N$ commutes with Hodge $*$ operator.}
\begin{equation}\label{divergence of D_N R}
D^\alpha D_N R_{\alpha\beta\gamma\delta} = J^{(N)}_{\beta\gamma\delta}, \quad D^\alpha {}\DNRstar_{\alpha\beta\gamma\delta} = \Jstar^{(N)}_{\beta\gamma\delta},
\end{equation}
where
\begin{equation}\label{JN}
J^{(N)}_{\beta\gamma\delta} = R^{\mu}{}_N{}_{\beta}{}^{\nu} R_{\mu}{}_{\nu}{}_{\gamma}{}_{\delta} + R^{\mu}{}_N{}_{\gamma}{}^{\nu} R_{\mu}{}_{\beta}{}_{\nu}{}_{\delta} + R^{\mu}{}_N{}_{\delta}{}^{\nu} R_{\mu}{}_{\beta}{}_{\gamma}{}_{\nu} + D^{\mu}N^{\nu} D_\nu R_{\mu\beta\gamma\delta},
\end{equation}
and
\begin{equation}\label{JNstar}
\JNstar_{\beta\gamma\delta} = R^{\mu}{}_N{}_{\beta}{}^{\nu} \Rstar_{\mu}{}_{\nu}{}_{\gamma}{}_{\delta} + R^{\mu}{}_N{}_{\gamma}{}^{\nu} \Rstar_{\mu}{}_{\beta}{}_{\nu}{}_{\delta} + R^{\mu}{}_N{}_{\delta}{}^{\nu} \Rstar_{\mu}{}_{\beta}{}_{\gamma}{}_{\nu}+ D^{\mu}N^{\nu} D_\nu \Rstar_{\mu\beta\gamma\delta}.
\end{equation}

\subsection{Short Pulse Ansatz and Scale Invariant Formulation}
We briefly recall the notions of \emph{signature} and \emph{scale} introduced by Klainerman and Rodnianski in \cite{K-R-09}. Let $\phi$ be either a null component of curvature or a connection coefficient, we use $N_a(\phi)$, $N_3(\phi)$ and $N_4(\phi)$ to denote the number of times $(e_a)_{i=1,2}$, respectively $e_3$ and $e_4$ appearing in the definition of $\phi$. The \emph{signature} of $\phi$, $sgn(\phi)$, and the \emph{scale} of $\phi$, $sc(\phi)$, are defined as
\begin{equation*}
sgn(\phi) =1\times N_4(\phi) + \frac{1}{2}\times N_a(\phi) + 0\times N_3(\phi)-1, ~sc(\phi) = -sgn(\phi)+ \frac{1}{2}.
\end{equation*}
We list the signatures and scales for all connection coefficients and curvature components,
\begin{center}
  \begin{tabular}{ | c | c | c || c | c | c || c | c | c |}
    \hline
         & signature & scale & &signature & scale & &signature & scale\\ \hline
    $\chi, \omega$ & 1 &$-\frac{1}{2}$ & $\alpha$ & 2 &$-\frac{3}{2}$ & $\alphab$ & 0 & $-\frac{1}{2}$ \\
    $\eta, \etab, \zeta$ & $\frac{1}{2}$ & 0 & $\beta$ & $\frac{3}{2}$ & $-1$ & $\betab$ &  $-\frac{1}{2}$ & 0\\
    $\chib,\omegab$ & 0 & $\frac{1}{2}$ &  $\rho, \sigma$ & 1 &$-\frac{1}{2}$ & & &\\
    \hline
  \end{tabular}
\end{center}
We impose following rules on signatures,
\begin{align*}
sgn(\nabla_4\phi) = sgn(\phi)+1,~ sgn(\nabla \phi) &= sgn(\phi)+\frac{1}{2},~ sgn(\nabla_3 \phi)= sgn(\phi)+0, \\
sgn(\phi_1 \cdot \phi_2) &= sgn(\phi_1)+sgn(\phi_2).
\end{align*}
We define the \emph{scale invariant norms} for $\phi$. Along null hypersurfaces $H_u^{(0,\ub)}$ or $\Hb_{\ub}^{(0,u)}$,
\begin{equation*}
 \|\phi\|_{L^2_{(sc)}(H_u^{(0,\ub)})} =  \delta^{-sc(\phi)-1}\|\phi\|_{L^2(H_u^{(0,\ub)})}, \quad \|\phi\|_{L^2_{(sc)}(\Hb_{\ub}^{(0,u)})} = \delta^{-sc(\phi)-\frac{1}{2}} \|\phi\|_{L^2(\Hb_{\ub}^{(0,u)})}.
\end{equation*}
On a two dimensional surface $S_{u,\ub}$,
\begin{equation*}
\|\phi\|_{L^p_{(sc)}(u,\ub)} = \|\phi\|_{L^p_{(sc)}(S_{u,\ub})} = \delta^{-sc(\phi)-\frac{1}{p}}  \|\phi\|_{L^p(S_{u,\ub})}.
\end{equation*}
Those norms are obviously related by formulas,
\begin{equation*}
 \|\phi\|^2_{L^2_{(sc)}(H_u^{(0,\ub)})} = \delta^{-1} \int_0^{\ub} \|\phi\|^2_{L^p_{(sc)}(u,\ub')} d\ub', \quad \|\phi\|^2_{L^2_{(sc)}(\Hb_{\ub}^{(0,u)})} = \int_0^{u} \|\phi\|^2_{L^p_{(sc)}(u',\ub)} d u'.
\end{equation*}
Those scale invariant norms come up naturally with a small parameter $\delta$. Roughly speaking, it captures the \emph{smallness} of the non-linear interaction. We have H\"older's inequality in scale invariant form,
\begin{equation}\label{Holder}
\|\phi_1 \cdot \phi_2 \|_{L^{p}_{(sc)}(S_{u,\ub})} \leq \delta^{\frac{1}{2}}\|\phi_1\|_{L^{p_1}_{(sc)}(S_{u,\ub})} \|\phi_2\|_{L^{p_2}_{(sc)}(S_{u,\ub})} ~\text{with}~ \frac{1}{p} =  \frac{1}{p_1} +  \frac{1}{p_2}.
\end{equation}
Similar estimates hold along null hypersurfaces.
\begin{remark} \label{gain or not gain} The rule of thumb for treating the nonlinear terms is, whenever one has a product of two terms,                                            \eqref{Holder} gains a $\delta^{\frac{1}{2}}$. We do have cases that \eqref{Holder} does not gain any power in $\delta$. In fact, if $f$ is a bounded (in usual sense) scalar function (say, bounded by a universal constant), the best we can hope is $
\|f \cdot \phi \|_{L^{p}_{(sc)}(S_{u,\ub})} \lesssim \|\phi\|_{L^{p}_{(sc)}(S_{u,\ub})}$.
In particular, in this paper, for $f$ can be $\tr\chib = \trchibt + \tr\chib_0$ where $\tr\chib_0 = \frac{4}{2r_0 + \ub - u} \sim 1$, we have to pay special attentions to the appearance of $\tr\chib$, see \cite{K-R-09} for more detailed descriptions.
\end{remark}
We introduce a family of scale invariant norms for connection coefficients where $p=2, 4$ or $\infty$, \footnote{\quad We use shorthand notations $\|(\psi,\psi',\psi'',\cdots)\| = \|\psi\|+\|\psi'\|+\|\psi''\|+ \cdots$}
\begin{align*}
\OSzerop(u,\ub) &=\delta^{\frac{1}{p}}\|(\chih, \chibh)\|_{L^p_{(sc)}(u,\ub)} + \|(\tr\chi, \omega, \eta, \etab, \trchibt, \omegab)\|_{L^p_{(sc)}(u,\ub)}, \\
\OSonep (u,\ub) &=\|\nabla(\chih, \tr\chi, \omega, \eta, \etab, \chibh, \trchibt, \omegab)\|_{L^p_{(sc)}(u,\ub)}, p \neq \infty,\\
\OHtwo(u,\ub)&=\|\nabla^2(\chih, \tr\chi, \omega, \eta, \etab, \chibh, \trchibt, \omegab)\|_{L^2_{(sc)}(H_u^{(0,\ub)})},\\
\OHbtwo(u,\ub)&=\|\nabla^2(\chih, \tr\chi, \omega, \eta, \etab, \chibh, \trchibt, \omegab)\|_{L^2_{(sc)}(\Hb_{\ub}^{(0,u)})},
\end{align*}
as well as for curvature components,
\begin{align*}
 \Rzero(u,\ub) & = \delta^{\frac{1}{2}}\|\alpha\|_{L^2_{(sc)}(H_u^{(0,\ub)})} + \|(\beta, \rho, \sigma, \betab)\|_{L^2_{(sc)}(H_u^{(0,\ub)})},\\
\Rzerob(u,\ub) & = \delta^{\frac{1}{2}}\|\beta\|_{L^2_{(sc)}(\Hb_{\ub}^{(0,u)})} + \|(\rho, \sigma, \betab, \alphab)\|_{L^2_{(sc)}(\Hb_{\ub}^{(0,u)})},\\
 \Rone(u,\ub) & = \delta^{\frac{1}{2}}\|\nabla_4 \alpha\|_{L^2_{(sc)}(H_u^{(0,\ub)})} + \|\nabla(\alpha, \beta, \rho, \sigma, \betab)\|_{L^2_{(sc)}(H_u^{(0,\ub)})},\\
\Roneb(u,\ub) & = \delta^{\frac{1}{2}}\|\nabla_3 \alphab\|_{L^2_{(sc)}(\Hb_{\ub}^{(0,u)})} + \|\nabla(\beta,\rho, \sigma,\betab, \alphab)\|_{L^2_{(sc)}(\Hb_{\ub}^{(0,u)})}.
\end{align*}

Finally, we introduce total norms. We define $\OSzerofour = \sup_{u,\ub} \OSzerofour(u,\ub)$ and $\Rzero = \sup_{u,\ub} \Rzero(u,\ub)$; similarly, we can define supremum norms for other scale invariant norms. The total norms are defined as follows,
\begin{align*}
 \mathcal{O} &= \Ozeroinfinity + \Ozerotwo +\Ozerofour + \Oonetwo + \Oonefour + \OHtwo + \OHbtwo,\\
 \R &= \Rzero+\Rone, \quad \Rb = \Rzerob + \Roneb.
\end{align*}

We use $\Oinitial$ and $\Rinitial$to denote total norms on the initial hypersurface $H_0$.

In the above definitions, all norms are scale invariant except for $\|\chih\|_{L^2_{(sc)}(u,\ub)}$, $\| \chibh \|_{L^2_{(sc)}(u,\ub)}$, $\|\alpha\|_{L^2_{(sc)}(H_u)}$,  $\|\nabla_3 \alphab\|_{L^2_{(sc)}(\Hb_{\ub})}$ and $\|\beta\|_{L^2_{(sc)}(\Hb_{\ub})}$. Those terms are understood to cause a loss of $\delta^{-\frac{1}{2}}$. Notice also $\beta$ on incoming hypersurfaces $\Hb_{\ub}$ is scale invariant. By abuse of language, we call those terms \emph{anomalies} or \emph{anomalous} if they cause a loss of $\delta^{-\frac{1}{p}}$ in $L^p_{(sc)}$ norm. Notice also all the connection coefficients are not anomalous in $L^\infty_{(sc)}$ norms. Inspired by this, we  use $\psi_g$ (`g` for good) to denote some (good) connection coefficient in $\{\tr\chi, \omega, \eta, \etab, \trchibt, \omegab\}$ and $\psi$ to denote an arbitrary connection coefficient. We use $\Psi_g$ to denote some curvature component in $\{\beta,\rho, \sigma,\betab, \alphab\}$ and $\Psi$ to denote an arbitrary connection coefficient. We shall also put a 'g' as a lower index to other quantities in order to indicate that this quantity is not anomalous. For example, we can write $\Psi(D_a R)_g = \alpha(D_a R)$ because $\|\alpha(D_a R)\|_{L^2_{(sc)}(H_u^{(0,\ub)})} \lesssim 1$.

\subsection{Main Result}
For Einstein equations with characteristic data prescribed on $\Hb_0$ (where the data is trivial) and $H_0$, we can freely specify the conformal geometry on $H_0$. In other words, we can specify $\chih$ freely along $H_0$ to fix an initial data for the evolutionary problem. We remark that, in contrast to the case where the initial data is give on a space-like hypersurface, there is no constraints and the data can be freely specified.

We require the initial data $\chi$ subject to the following norm is finite,
\begin{align*}
\Izero &= \delta^{\frac{1}{2}}\|\chih\|_{L^{\infty}(H_0)} + \sup_{0\leq \ub \leq \delta} [\sum_{k=0}^{2}\delta^{\frac{1}{2}}\|(\delta \nabla_4)^k \chih \|_{L^2(S_{0,\ub})}\\
 &\quad + \sum_{k=0}^{1}\sum_{m=0}^3 \delta^{\frac{1}{2}}\|(\delta^{\frac{1}{2}}\nabla)^m(\delta \nabla_4)^k \chih\|_{L^2(S_{0,\ub})}],
\end{align*}
i.e. our main assumption on the initial data is the following short pulse ansatz,
\begin{equation}\label{initial_ansatz_1}
\Izero < \infty.
\end{equation}
Notice that the size of $\chih$ can be as large as $\delta^{-\frac{1}{2}}$ so the problem is far away from small data regime.

The ansatz \eqref{initial_ansatz_1} was introduced by Klainerman and Rodnianski in \cite{K-R-09}. This initial data set is larger than those of Christodoulou's original short pulses. In fact, \eqref{initial_ansatz_1} allows more components than those of Christodoulou's to be as large as $\delta^{-\frac{1}{2}}$. In this ansatz, we allow $\nabla$ behaves as $\delta^{-\frac{1}{2}}$; in Christodoulou's ansatz, $\nabla$ behaves as $1$.

Under the ansatz \eqref{initial_ansatz_1}, with the help of null structure equations, we can easily derive the following estimates on initial outgoing surfaces $H_0$,
\begin{lemma}Under the ansatz \eqref{initial_ansatz_1}, along the initial outgoing hypersurface $H_0$, if $\delta$ is sufficiently small, there is a constant $c(\Izero)$ depending only on $\Izero$, such that
\begin{equation}\label{initial_ansatz_1_prime}
\Oinitial +\Rinitial   \lesssim c(\Izero).
\end{equation}
\end{lemma}
Thanks to this proposition, we shall replace \eqref{initial_ansatz_1} by \eqref{initial_ansatz_1_prime}. And we omit the proof and refer the reader to \cite{K-R-09} or Chapter 2 of \cite{Ch}. The next proposition claims that we can control connection coefficients provided that we have bound on curvatures. This is \textit{Theorem A} in \cite{K-R-09} and we omit the proof.
\begin{proposition}\label{theoremA}
Assume that $\Ozero$, $\R$ and $\Rb$ are finite in $\D(1,\delta)$. Then there is a constant $C$ depending only on $\Ozero, \R$ and $\Rb$ such that \footnote{\quad Throughout the paper, we use $C$ to denote a constant depending only on $\Ozero, \R$ and $\Rb$.}
\begin{equation}\label{O norm}
\O \lesssim C.
\end{equation}
\end{proposition}

We now state our main theorem. This is the following propagation estimates, which asserts that if \eqref{initial_ansatz_1_prime} holds on initial hypersurface $H_0$, thus on the whole $\D(1,\delta)$ we can bound the curvature norms $\R$ and $\Rb$ by a function depending only on initial data.
\begin{theoremC}\label{theoremC}
 Assume the short pulse ansatz \eqref{initial_ansatz_1} hence \eqref{initial_ansatz_1_prime}, thus if $\delta$ is sufficiently small, we have
 \begin{equation}\label{main estimates}
 \R + \Rb  \lesssim c(\Izero),
 \end{equation}
where $c(\Izero)$ is a constant depending only on the size of the initial data.
\end{theoremC}

The main consequence of our estimates is the following global existence theorem,
\begin{theoremD}\label{theoremD}
 Given initial data $\chih$ satisfying \eqref{initial_ansatz_1}, if $\delta$ is sufficiently small, we can construct an unique solution of the Einstein vacuum equations
 \begin{equation*}
 R_{\alpha\beta} = 0,
 \end{equation*}
 on the whole region $\D(1,\delta)$.
\end{theoremD}

Once we have the Main Estimates, the proof of the theorem is a typical continuity argument. We refer the reader to Chapter 16 of \cite{Ch} for the detailed proof. We shall not repeat this argument.

We want to remark that the formation of trapped surfaces can not be derived from the initial data ansatz \eqref{initial_ansatz_1}. We have to add two more modifications to make sure that $H_0$ is free of trapped surface and $S(1,\delta)$ is a trapped surface. More precisely, in addition to \eqref{initial_ansatz_1_prime}, we also assume that
\begin{equation}\label{initial_ansatz_2}
\sup_{0\leq \ub \leq \delta}\sum_{k=2}^4 \delta^{\frac{1}{2}}\|(\delta^{\frac{1}{2}})^k \nabla^k \chih\|_{L^2(S_{0,\ub})} \leq \varepsilon,
\end{equation}
for sufficiently small $\varepsilon$ such that $0< \delta \ll \varepsilon$ and we also assume $\chih$ satisfies,
\begin{equation}\label{initial_ansatz_3_formation_ansatz}
 (1+C_0 \delta^{\frac{1}{2}})\frac{2(r_0-u)}{r_0^2}<\int_0^\delta |\chih(0,\ub)|^2 d\ub < \frac{2(r_0-\delta)}{r_0^2},
\end{equation}
where $C_0$ is a universal constant and $r_0 \sim 2$ measure the maximal radius of the flat part of $H_0$. These condition guarantees the dynamical formation of trapped surfaces, we refer the readers to \cite{K-R-09} for details.

\subsection{The Bootstrap Argument}\label{structure of the proof}

We use a bootstrap argument to prove the \textbf{Main Estimates}. To be more precisely, we assume that $\R$ and $\Rb$ are finite. This assumption holds initially near $H_0$. Our main task is to carry out the following estimates under this assumption,
\begin{equation}\label{final inequality}
 \R + \Rb \lesssim c(\Izero) (1+ (\R+ \Rb)^{\frac{7}{8}}) +C\delta^{\frac{1}{32}}.
\end{equation}
Thus, if $\delta$ is sufficiently small, we have
\begin{equation*}
 \R + \Rb \lesssim c(\Izero).
\end{equation*}
This yields the \textbf{Main Estimates} \eqref{main estimates}.

The derivation for \eqref{final inequality} consists of two steps. The first step is to derive energy estimates for curvature components; the second step is to derive energy estimates for one derivatives of curvature components. The first step is done in \cite{K-R-09}:
\begin{proposition}[Proposition 14.9 in \cite{K-R-09}]\label{energy estimates for curvature}
If $\delta$ is sufficiently small, we have
\begin{align*}
\|\alpha\|_{L^2_{(sc)}(H_u)} + \|\beta\|_{L^2_{(sc)}(\Hb_{\ub})} &\lesssim \delta^{-\frac{1}{2}}\Izero + C\delta^{\frac{1}{4}},\\
\|(\beta, \rho, \sigma, \betab) \|_{L^2_{(sc)}(H_u)} + \|(\rho, \sigma, \betab, \alphab)\|_{L^2_{(sc)}(\Hb_{\ub})} &\lesssim \Izero + c(\Izero)\R^{\frac{1}{2}}+C\delta^{\frac{1}{8}}.
\end{align*}
\end{proposition}

As we explained in the introduction, our main target is the second step, i.e. to derive energy estimates for one derivatives of curvature components in the form of \eqref{final inequality} without constructing rotational vector fields $\O^{(i)}$'s. This is done in the rest of the paper.

\section{Preliminary Estimates}
In this section, we collect some estimates either already derived in \cite{K-R-09} or relatively easy to prove directly.

\subsection{Improved Estimates on Curvature}
We have $L^2$ estimates for curvature on each leaf $S = S_{u,\ub}$.
\begin{proposition}[Proposition 6.6, Proposition 6.9 in \cite{K-R-09}]\label{L2 estimates for curvature}
If $\delta$ is sufficiently small, we have
\begin{equation*}
\delta^{\frac{1}{2}}\|\alpha\|_{L^2_{(sc)}(S)} + \|(\beta,\rho,\sigma,\betab,\alphab)\|_{L^2_{(sc)}(S)} \lesssim C.
\end{equation*}
\end{proposition}
We also have $L^4_{(sc)}$ estimates on curvature components. We recall again that the constant $C$ depends on $\R$ and $\Rb$ which has information on the one horizontal derivatives of the curvature.
\begin{proposition}[Lemma 10.1 in \cite{K-R-09}]\label{L4 estimates for curvature}
If $\delta$ is sufficiently small, we have
\begin{equation*}
\delta^{\frac{1}{4}}\|\alpha\|_{L^4_{(sc)}(S)} + \|(\beta,\rho,\sigma,\betab,\alphab)\|_{L^4_{(sc)}(S)} \lesssim C.
\end{equation*}
\end{proposition}
This proposition is proved by virtue of null Bianchi equations \eqref{NBE_Lb_alpha}-\eqref{NBE_L_alphab} and following scale invariant Sobolev trace type inequalities stated as Proposition 4.15 in \cite{K-R-09}.
\begin{proposition}\label{Sobolev Trace} Given an arbitrary tensor field $\phi$, we have
\begin{equation*}
 \|\phi\|_{L^4_{(sc)}(S)} \lesssim (\delta^{\frac{1}{2}}\|\phi\|_{L^2_{(sc)}(H_u)}+\|\nabla \phi\|_{L^2_{(sc)}(H_u)})^{\frac{1}{2}}(\delta^{\frac{1}{2}}\|\phi\|_{L^2_{(sc)}(H_u)}+\|\nabla_4 \phi\|_{L^2_{(sc)}(H_u)})^{\frac{1}{2}},
\end{equation*}
\begin{equation*}
 \|\phi\|_{L^4_{(sc)}(S)} \lesssim (\delta^{\frac{1}{2}}\|\phi\|_{L^2_{(sc)}(\Hb_{\ub})}+\|\nabla \phi\|_{L^2_{(sc)}(\Hb_{\ub})})^{\frac{1}{2}}(\delta^{\frac{1}{2}}\|\phi\|_{L^2_{(sc)}(\Hb_{\ub})}+\|\nabla_3 \phi\|_{L^2_{(sc)}(\Hb_{\ub})})^{\frac{1}{2}}.
\end{equation*}
\end{proposition}
Integrating along null hypersurfaces, we derive
\begin{corollary}\label{L4 estimates for curvature on null hypersurfaces}
If $\delta$ is sufficiently small, we have
\begin{align*}
\delta^{\frac{1}{4}}\|\alpha\|_{L^4_{(sc)}(H_u)}& + \|(\beta,\rho,\sigma,\betab,\alphab)\|_{L^4_{(sc)}(H_u)} \lesssim C, \\
\delta^{\frac{1}{4}}\|\alpha\|_{L^4_{(sc)}({\Hb}_{\ub})}& + \|(\beta,\rho,\sigma,\betab,\alphab)\|_{L^4_{(sc)}({\Hb}_{\ub})} \lesssim C.
\end{align*}
\end{corollary}

We also collect some more precise estimates on $\chih$, $\chibh$ and $\omega$. They are essentially developed in Section 5 of \cite{K-R-09} provided we bound $\alpha$ by Proposition \ref{energy estimates for curvature}.
\begin{proposition}\label{pricise on chi chib}
If $\delta$ is sufficiently small, we have
\begin{align*}
 \|\chih\|_{L_{(sc)}^4(S)} \lesssim & \delta^{-\frac{1}{4}} c(\Izero) \R^{\frac{1}{2}} + C \delta^\frac{1}{8}, ~\|\chibh\|_{L_{(sc)}^4(S)} \lesssim  \delta^{-\frac{1}{4}}\Izero + \Rzerob^{\frac{1}{2}}\R^{\frac{1}{2}} + C \delta^\frac{1}{2},\\
\|\chih\|_{L^2_{(sc)}(H_u)} \lesssim & \delta^{-\frac{1}{2}} \Izero + C\delta^{\frac{1}{4}}, ~~~\|\omega\|_{L_{(sc)}^4(S)} \lesssim \Rzero^\frac{1}{2} \Rone^{\frac{1}{2}} + C\delta^{\frac{1}{4}}.
\end{align*}
\end{proposition}

\subsection{Comparison Estimates}\label{comparison}
The aim is to compare null components of derivatives of curvature and derivatives of null components of curvature. The first lemma compares $\nabla_N \Psi$ with $\Psi(D_N R)$ where $N =L$ or $\Lb$.
\begin{lemma}\label{comparison lemma DN commute with Psi}
For $N=L$ or $\Lb$, we have
\begin{equation*}
\|\Psi(D_N R)-\nabla_N \Psi\|_{L^2_{(sc)}(H)}+ \|\Psi(D_N R)-\nabla_N \Psi\|_{L^2_{(sc)}(\Hb)}  \lesssim  C \delta^{\frac{1}{4}}.
\end{equation*}
\end{lemma}
\begin{proof}
By direct computations, in schematic form, we have the following commutator formula,\footnote{\quad The term $\psi_g \cdot \Psi$ denotes a sum of products of this form. We recall once again that $\psi$ denote a generic connection coefficient and $\Psi$ denote a generic curvature component. The subindex $g$ means this component is not anomalous.}
\begin{equation*}
\Psi(D_N R)-\nabla_N \Psi = \psi_g \cdot \Psi
\end{equation*}
To estimate $\psi_g \cdot \Psi$ on some	 null hypersurface (say $H$ without loss of generality), we use $L^4_{(sc)}$ estimates on both factors and notice that $\Psi$ may cause a loss of $\delta^{\frac{1}{4}}$, thus we have
\begin{equation*}
\|\Psi(D_N R)-\nabla_N \Psi\|_{L^2_{(sc)}(H)} \lesssim \delta^{\frac{1}{2}}\|\psi_g\|_{L^4_{(sc)}(H)}\|\Psi\|_{L^4_{(sc)}(H)} \lesssim C \delta^{\frac{1}{4}},
\end{equation*}
since $\|\psi_g\|_{L^4_{(sc)}(H)} \lesssim C$ and $\|\Psi\|_{L^4_{(sc)}(H)} \lesssim C \delta^{-\frac{1}{4}}$. This completes the proof.
\end{proof}
The next lemma, by investigating null Bianchi equations, compares $\nabla_N \Psi$ with $\nabla \Psi'$.
\begin{lemma}\label{lemma estimates on nabla_N Psi}
We have following estimates,\\
(1). For $\nabla_3 \alpha$, we have
\begin{equation*}
\|\nabla_3 \alpha\|_{L^2_{(sc)}(H_u)} \lesssim \Izero \,\delta^{-\frac{1}{2}}+\Rone + C\delta^{\frac{1}{4}}, ~ \|\nabla_3 \alpha\|_{L^2_{(sc)}({\Hb}_{\ub})} \lesssim  \Izero \,\delta^{-\frac{1}{2}} + \Roneb + C\delta^{\frac{1}{4}},
\end{equation*}
(2). For $\nabla_N \Psi \notin \{\nabla_3 \alpha, \nabla_4 \alpha, \nabla_3 \alphab\}$, except $\nabla_4 \beta$ on ${\Hb}_{\ub}$ and $\nabla_3 \betab$ on $H_u$, we have
\begin{equation*}
\|\nabla_N \Psi\|_{L^2_{(sc)}(H_u)} \lesssim \Rzero + \Rone + C\delta^{\frac{1}{4}} , \quad \|\nabla_N \Psi\|_{L^2_{(sc)}({\Hb}_{\ub})} \lesssim \Rzerob + \Roneb + C\delta^{\frac{1}{4}},
\end{equation*}
\end{lemma}
\begin{remark}
We can not estimate on $\|\nabla_4 \beta\|_{L^2_{(sc)}({\Hb}_{\ub})}$ or $\|\nabla_3 \betab\|_{L^2_{(sc)}(H_u)}$ at the moment. The reason is that we do not have any knowledge on either $\|\nabla \alpha\|_{L^2_{(sc)}({\Hb}_{\ub})}$ or $\|\nabla \alphab\|_{L^2_{(sc)}(H_u)}$. But they are controlled in next section when we derive estimates on $\nabla_4 \alpha$ and $\nabla_3 \alphab$.
\end{remark}

\begin{proof}
We only sketch one estimate (which may be regarded as the most difficult one). The rest is similar so we skip. For $\nabla_3 \alpha$, according to \eqref{NBE_Lb_alpha}, we have
\begin{align*}
 \nabla_3 \alpha &=- \frac{1}{2}\tr \chib \alpha + \nabla \tensor \beta + 4\omegab \alpha - 3(\chih \rho + ^*\!\chih \sigma)+(\zeta + 4\eta)\tensor \beta\\
 &=\tr\chib_0 \cdot \alpha + \nabla \beta + \psi_g \cdot \alpha + \psi\cdot \Psi_g,
\end{align*}
We can bound $\nabla \beta$ by $\Roneb$ and bound last two terms by $C \delta^{\frac{1}{4}}$ thanks to H\"older's inequality. For $\tr\chib_0 \alpha$, in view of Proposition \ref{energy estimates for curvature}, it is bounded by $\Izero \delta^{-\frac{1}{2}}+ C \delta^{\frac{1}{4}}$. Combining all the above, we have the desired estimates for $\nabla_3 \alpha$. This completes the proof.
\end{proof}
Combining previous two lemmas, we derive the comparison estimates for null components of $D_3 R_{\alpha\beta\gamma\delta}$ and $D_4 R_{\alpha\beta\gamma\delta}$.
\begin{proposition}\label{comparison proposition}
(1). For $\alpha(D_3 R)$, we have
\begin{equation*}
\|\alpha(D_3 R)\|_{L^2_{(sc)}(H_u)} \lesssim \Izero \, \delta^{-\frac{1}{2}} + \Rone + C\delta^{\frac{1}{4}}, ~ \|\alpha(D_3 R)\|_{L^2_{(sc)}({\Hb}_{\ub})} \lesssim \Izero \, \delta^{-\frac{1}{2}} +\Roneb + C\delta^{\frac{1}{4}},
\end{equation*}
(2). For $\Psi(D_N R) \notin \{\alpha(D_3 R), \alpha(D_4 R), \alphab(D_3 R)\}$, except $\beta(D_4 R)$ on ${\Hb}_{\ub}$ and $\betab(D_3 R)$ on $H_u$, we have
\begin{equation*}
\|\Psi(D_N R)\|_{L^2_{(sc)}(H_u)} \lesssim \Rzero + \Rone + C\delta^{\frac{1}{4}} , \quad \|\Psi(D_N R)\|_{L^2_{(sc)}({\Hb}_{\ub})} \lesssim \Rzerob + \Roneb + C\delta^{\frac{1}{4}}.
\end{equation*}
\end{proposition}

We now consider horizontal derivatives and compare $\nabla_c \Psi$ with $\Psi(D_c R)$.
\begin{proposition}\label{comparison proposition horizontal version}
If $\Psi \neq \beta$, we have
\begin{align*}
\|\Psi(D_c R)-\nabla_c \Psi\|_{L^2_{(sc)}(H)} &+ \|\Psi(D_c R)-\nabla_c\Psi\|_{L^2_{(sc)}(\Hb)}  \lesssim \R + \Rb + C \delta^{\frac{1}{4}},\\
\|\Psi(D_c R)\|_{L^2_{(sc)}(H)} &+ \|\Psi(D_c R)\|_{L^2_{(sc)}(\Hb)}  \lesssim \R + \Rb + C \delta^{\frac{1}{4}}.
\end{align*}
\end{proposition}
\begin{proof}
The proof is quite similar to Lemma \ref{comparison lemma DN commute with Psi}. Notice that, if $\Psi \neq \beta$, then
\begin{equation*}
\Psi(D_c R)-\nabla_c \Psi = \psi_g \cdot \Psi + \tr\chib_0 \cdot \Psi_g.
\end{equation*}
The product term is bounded by $C\delta^{\frac{1}{4}}$ as before and last term is bounded by $\R + \Rb$ by definition. Thus we have the first inequality. The second follows immediately.
\end{proof}
We turn to the anomalous term $\nabla_c \beta$. First of all, notice that
\begin{align*}
 \beta(D_c R)_a &= \nabla_c \beta_a - \frac{1}{2}\chib_{cb} \alpha_{ba} + \zeta_c \cdot \beta_a - \frac{1}{2}\chi_{c a}\cdot( 3\rho - \sigma)\\
&=  \tr\chib_0 \cdot \alpha+ \nabla \beta + \psi_g \cdot \Psi + \tr\chib_0 \cdot \Psi_g
\end{align*}
Thus, the first term $\tr\chib_0 \cdot \alpha$ causes anomalous behavior. Hence we have
\begin{equation}
\|\beta(D_c R)\|_{L^2_{(sc)}(H)} + \|\beta(D_c R)\|_{L^2_{(sc)}(\Hb)}  \lesssim \Izero \cdot \delta^{-\frac{1}{2}} + \R + \Rb + C \delta^{\frac{1}{4}}.
\end{equation}

\subsection{Mild Anomalies}\label{mild anomaly}
We see that $\alpha(D_4 R)$, $\alphab(D_3 R)$, $\alpha(D_3 R)$ and $\beta(D_c R)$ are \textit{anomalies} since their estimates cause a loss of $\delta^{-\frac{1}{2}}$. We shall distinguish  $\alpha(D_3 R)$ and $\beta(D_c R)$ from $\alpha(D_4 R)$ and $\alphab(D_3 R)$. Heuristically, $\alpha(D_3 R)$ and $\beta(D_c R)$ are anomalous in zero order;  $\alpha(D_4 R)$ and $\alphab(D_3 R)$ are anomalous in first order.

We observe that $\alpha(D_3 R)$ and $\beta(D_c R)$ are anomalous because they can be written schematically as
\begin{equation}
\tr\chib_0 \cdot \alpha + \Psi_g + (\nabla \Psi)_g + \psi_g \cdot \Psi.
\end{equation}
This amounts to say that their anomalies  come from the zero order term $\tr\chib_0 \cdot \alpha$. We shall call $\alpha(D_3 R)$ and $\beta(D_c R)$ \textit{mild anomalies}. By the above comparison lemmas, $\nabla_3 \alpha$ is also a mild anomaly.

In what follows, we use this observation to replace $\alpha(D_3 R)$ and $\beta(D_c R)$ by the worst term $\alpha$. We have two advantages of performing this substitution: the first is that we have already controlled $\alpha$ in Proposition \ref{energy estimates for curvature}; the second is that we can use $L^4_{(sc)}$ estimates on $\alpha$ (but not on either $\alpha(D_3 R)$ or $\beta(D_c R)$). The other terms are much easier to control in all the contexts since they are not anomalous at all. This becomes much clear and important when we derive estimates in following sections.

\section{Anomalous Estimates}
\subsection{Estimates on $\nabla_4 \alpha$}
We take $X=Y=Z=L$ in \eqref{basic energy identity} and $N=e_4$ in \eqref{divergence of D_N R} to derive
\begin{equation*}
\int_{H_u} |\alpha(D_4 R)|^2 + \int_{{\Hb}_{\ub}}|\beta(D_4 R)|^2 \lesssim \int_{H_0} |\alpha(D_4 R)|^2 + \doubleint_{\D(u,\ub)} D^\mu Q[D_4 R]_{\mu444} + {}^{(L)}\!\pi_{\mu\nu}Q^{\mu\nu}{}_{44}.
\end{equation*}
In view of Lemma \ref{comparison lemma DN commute with Psi}, by ignoring the term $\beta(D_4 R)$ on the left hand side and passing to scale invariant norms, we derive
\begin{align*}
 \|\nabla_4 \alpha \|_{L^2_{(sc)}(H_u^{(0,\ub)})}^2 &\lesssim \|\nabla_4 \alpha \|_{L^2_{(sc)}(H_0^{(0,\ub)})}^2 +C\delta^{\frac{1}{2}}+ \delta^{3}\doubleint_{\D(u,\ub)} D^\mu Q[D_4 R]_{\mu444} + {}^{(L)}\!\pi_{\mu\nu}Q^{\mu\nu}{}_{44}\\
 &\lesssim \delta^{-1}\Izero^2 + (I +J) + K.
\end{align*}
We split the bulk integral into three terms $I$, $J$ and $K$ in such a way that the integrands of $I+J$ are exactly the terms appearing in $D^\mu Q[D_4 R]_{\mu444}$ and the integrands of $K$ are exactly the terms appearing in ${}^{(L)}\!\pi_{\mu\nu}Q^{\mu\nu}{}_{44}$.
\begin{remark}\label{remark forget delta}
In the above terms $I$, $J$ and $K$, there is a $\delta$ to certain power in front of the integral which reflects the fact that the inequality is scale invariant. In what follows, this phenomenon appears constantly. To simplify expressions, we shall always ignore this $\delta$ to certain power since they can be easily restored by investigating the signatures.
\end{remark}
According to \eqref{divergence of Q} and \eqref{divergence of D_N R}, the integrand in $I + J$ has the following schematic form,
\begin{equation*}
D_4R_{4}{}^{\mu}{}_{4}{}^{\nu} \Jfour_{\mu4\nu} + D_4\Rstar_{4}{}^{\mu}{}_{4}{}^{\nu} \,\Jfourstar_{\mu4\nu} = \Psi(D_4 R)^{(s_0)} \cdot \Psi^{(s_1)} \cdot \Psi^{(s_2)} + \psi^{(s_0)} \cdot \Psi(D_4 R)^{(s_1)} \cdot  \Psi(D_\nu R)^{(s_2)}
\end{equation*}
with $s_0 + s_1 +s_2 =6$. Here we use the upper index $s_i$'s to indicate the signature. We require the integrand of $I$ is of the form $\Psi(D_4 R)^{(s_0)} \cdot \Psi^{(s_1)} \cdot \Psi^{(s_2)}$ and the integrand of $J$ is of the form $\psi^{(s_0)} \cdot \Psi(D_4 R)^{(s_1)} \cdot  \Psi(D_\nu R)^{(s_2)}$. We remark that $\psi^{(s_0)} \neq \tr\chib_0$. \footnote{\quad In following sections, we shall always split the error integrals into three terms $I$, $J$ and $K$ in the same way. And we shall not explain when one encounters this situation later.}

We now examine the possible anomalies in the integrand of $I$. For $\Psi(D_4 R)^{(s_0)}$ on $H_u$, according to Lemma \ref{comparison proposition}, the only possible anomaly is $\alpha(D_4 R)$ which has $3$; for $\Psi^{s_1}$ and $\Psi^{s_2}$, the only possible anomaly is $\alpha$ with signature $2$. Thus, in view of the signature consideration, at least one of the terms is not anomalous. We say that, in this situation, we do not have \emph{triple anomalies}.

In view of Corollary \ref{L4 estimates for curvature on null hypersurfaces}, we estimate $I$ as follows\footnote{ \quad In this section, since we deal with exceptional terms, we can be extremely wasteful. In other sections, to derive energy estimates, we have to exploit the cancelation hidden in the integrand.}
\begin{align*}
I &\lesssim \delta^{\frac{1}{2}}\int_0^{u} \| \Psi(D_4 R)^{(s_0)}\|_{L^2_{(sc)}({H}_{u'})}\| \Psi^{(s_1)}\|_{L^4_{(sc)}({H}_{u'})}\|\Psi^{(s_2)}_g\|_{L^4_{(sc)}({H}_{u'})}d u'\\
& \quad + \delta^{\frac{1}{2}}\int_0^{u} \|\Psi(D_4 R)^{(s_0)}_g\|_{L^2_{(sc)}({H}_{u'})}\| \Psi^{(s_1)}\|_{L^4_{(sc)}({H}_{u'})}\|\Psi^{(s_2)}\|_{L^4_{(sc)}({H}_{u'})}d u'\\
&\lesssim \delta^{\frac{1}{2}}\int_0^{u}C^3 \delta^{-\frac{1}{2}}\delta^{-\frac{1}{4}}d u' + \delta^{\frac{1}{2}}\int_0^{u}C^3 \delta^{-\frac{1}{4}}\delta^{-\frac{1}{4}}d u' \lesssim C\delta^{-\frac{1}{4}}.
\end{align*}

For $K$, according to signature considerations, we can also avoid triple anomalies. Thus,
\begin{equation*}
{}^{(L)}\!\pi_{\mu\nu}\,Q^{\mu\nu}{}_{44} = \psi \cdot \alpha(D_4 R) \cdot \alpha(D_4 R) + \psi \cdot \Psi_g(D_4 R) \cdot \Psi(D_4 R).
\end{equation*}
We dominate $\psi_g$ and $\psi$ in $L_{(sc)}^{\infty}$ norm to derive
\begin{align*}
K &\lesssim C\delta^{\frac{1}{2}}\int_0^{u}\|\alpha(D_4 R)\|_{L^2_{(sc)}({H}_{u})}^2 + C\delta^{\frac{1}{2}}\int_0^{u} \|\Psi(D_4 R)\|_{L^2_{(sc)}({H}_{u})}\|\Psi_g(D_4 R)\|_{L^2_{(sc)}({H}_{u})}\\
&\lesssim  C\delta^{\frac{1}{2}}\int_0^{u}C\delta^{-\frac{1}{2}} C \delta^{-\frac{1}{2}}d u' + C\delta^{\frac{1}{2}}\int_0^{u} C^2 \delta^{-\frac{1}{2}}d u'\lesssim C \delta^{-\frac{1}{2}}.
\end{align*}

For $J$, it possesses most of the features of $K$. We still need to investigate the structure of the integrand. In fact, the integrand is of the following form \footnote{\quad We ignore the terms coming from the Hodge dual of curvatures since due to the symmetry. They can be treated exactly in the same way. This remark applies to all situations with one exception at the end of the paper. }
\begin{equation}\label{precise expression 1}
 D_4 R_{4}{}^{\rho}{}_{4}{}^{\delta} \cdot  D^\mu L^\nu  \cdot D_\nu R_{\mu \rho 4 \delta} = \psi^{(s_0)} \cdot \Psi(D_4 R)^{(s_1)} \cdot  \Psi(D_\nu R)^{(s_2)}
\end{equation}
Thus, since $\psi^{(s_0)} \neq \tr\chib_0$, we can bound $J$ by
\begin{align*}
J &\lesssim C\delta^{\frac{1}{2}}\int_0^{u} \|\Psi(D_4 R)\|_{L^2_{(sc)}({H}_{u})}\|\Psi (D_\nu R)\|_{L^2_{(sc)}({H}_{u})} \lesssim  C\delta^{\frac{1}{2}}\int_0^{u}C\delta^{-\frac{1}{2}} C \delta^{-\frac{1}{2}}\lesssim C \delta^{-\frac{1}{2}}.
\end{align*}

Combining the estimates for $I$, $J$, $K$ and Young's inequality,  we derive
\begin{equation*}
 \|\nabla_4 \alpha \|_{L^2_{(sc)}(H_u^{(0,\ub)})}^2 \lesssim \delta^{-1}\Izero^2 + C\delta^{-\frac{1}{2}} \lesssim \delta^{-1}c(\Izero) + C \delta^{\frac{1}{2}}.
\end{equation*}

In conclusion, we have the desired anomalous estimates on $\nabla_4 \alpha$
\begin{equation}\label{estimates for nabla_4 alpha on H}
 \|\nabla_4 \alpha \|_{L^2_{(sc)}(H_u^{(0,\ub)})} \lesssim \delta^{-\frac{1}{2}}c(\Izero) + C\delta^{\frac{1}{4}}.
\end{equation}
\begin{remark}
The above analysis also yields the following estimates which is absent in \emph{Proposition \ref{comparison proposition}},
\begin{equation}\label{estimates for nabla_4 beta on Hb}
\|\nabla_4 \beta \|_{L^2_{(sc)}({\Hb}_{\ub}^{(0,u)})}  \lesssim \delta^{-\frac{1}{2}}\Izero + C\delta^{\frac{1}{4}}.
\end{equation}
\end{remark}

\subsection{Estimates on $\nabla_3 \underline{\alpha}$}
We take $N=\Lb$ in \eqref{divergence of D_N R} and $X=Y=Z=\Lb$ in \eqref{basic energy identity} to derive
\begin{equation*}
\int_{H_u} |\betab(D_3 R)|^2 + \int_{{\Hb}_{\ub}}|\alphab(D_3 R)|^2 \lesssim \int_{H_0} |\betab(D_3 R)|^2 + \doubleint_{\D(u,\ub)} D^\mu Q[D_3 R]_{\mu333} + {}^{(\Lb)}\!\pi_{\mu\nu}Q^{\mu\nu}{}_{33}.
\end{equation*}
In view of Lemma \ref{comparison lemma DN commute with Psi}, by ignoring the term $\beta(D_4 R)$ on the left hand side and passing to scale invariant norms, we derive
\begin{align*}
 \|\nabla_3 \alpha \|_{L^2_{(sc)}(H_u^{(0,\ub)})}^2 &\lesssim \delta^{-1}\Izero +C\delta^{\frac{1}{2}}+ \delta^{-1}\doubleint_{\D(u,\ub)} D^\mu Q[D_3 R]_{\mu333} + {}^{(\Lb)}\!\pi_{\mu\nu}Q^{\mu\nu}{}_{33}\\
 &\lesssim \delta^{-1}\Izero^2 + I +J+ K.
\end{align*}
We split the bulk integral into $I$, $J$ and $K$ in the obvious way as before. The integrands in $I$ have the following schematic form,
\begin{equation*}
\Psi(D_3 R)^{(s_0)} \cdot \Psi^{(s_1)} \cdot \Psi^{(s_2)}
\end{equation*}
with $s_0 + s_1 +s_2 =1$. According to signature considerations, none of $\Psi^{(s_1)}$ and $\Psi^{(s_3)}$ can be anomalous since the signature for anomalous curvature component has signature $2$. Thus, we estimate $I$ as follows
\begin{align*}
I  &\lesssim \delta^{\frac{1}{2}}\delta^{-1}\int_0^{\ub} \|\Psi(D_3 R)^{(s_0)}\|_{L^2_{(sc)}({\Hb}_{\ub'})}\|\Psi^{(s_1)}_g\|_{L^4_{(sc)}({\Hb}_{\ub'})}\|\Psi^{(s_2)}_g\|_{L^4_{(sc)}({\Hb}_{\ub'})}d \ub' \\
&\lesssim \delta^{\frac{1}{2}}\delta^{-1}\int_0^{\ub}C \delta^{-\frac{1}{2}}C^2 d \ub'\lesssim C.
\end{align*}

We now deal with $K$. Again, by counting signatures, the anomalous terms $\alpha(D_4 R)$, $\alpha(D_3 R)$ and $\beta(D_c R)$ do not appear in the integrand since their signatures are at least $2$. Thus the only possible anomaly is $\alphab(D_3 R)$. Hence, if their is a double anomaly in curvature components, they must be two $\alphab(D_3 R)$. This forces the connection coefficient to have signature $1$. In particular, it can not be $\tr\chib_0$. Thus, the integrands have the following schematic form \footnote{\quad Recall that $\tr\chib_0$ is a constant with worse estimates in $L^\infty_{(sc)}$ than usual connection coefficients.}
\begin{equation*}
{}^{(\Lb)}\!\pi_{\mu\nu} \, Q^{\mu\nu}{}_{33} = (\psi + \tr\chib_0) \cdot \Psi_g(D_3 R) \cdot \Psi(D_3 R) + \psi \cdot \alphab(D_3 R) \cdot \alphab(D_3 R)
\end{equation*}
Hence, bounding $\psi$ in $L^\infty_{(sc)}$ yields,
\begin{align*}
K &\lesssim (C\delta^{\frac{1}{2}}+1) \delta^{-1}\int_0^{\ub} \|\Psi(D_3 R)\|_{L^2_{(sc)}({\Hb}_{\ub})}\|\Psi_g(D_3 R)\|_{L^2_{(sc)}({\Hb}_{\ub})} \\
&\quad + C\delta^{\frac{1}{2}} \cdot \delta^{-1}\int_0^{\ub} \|\alphab(D_3 R)\|_{L^2_{(sc)}({\Hb}_{\ub})}\|\alphab(D_3 R)\|_{L^2_{(sc)}({\Hb}_{\ub})}\lesssim C \delta^{-\frac{1}{2}}.
\end{align*}

The analysis for $J$ is exactly the same as $K$, hence, we have
\begin{align*}
J \lesssim C \delta^{-\frac{1}{2}}.
\end{align*}

Combining those estimates for $I$, $J$ and $K$, as we have done for $\nabla_4 \alpha$ in last subsection, we derive
\begin{equation}\label{estimates for nabla_3 alphab on Hb}
 \|\nabla_3 \alphab \|_{L^2_{(sc)}({\Hb}_{\ub}^{(0,u)})} \lesssim \delta^{-\frac{1}{2}}\Izero + C\delta^{\frac{1}{4}}.
\end{equation}
\begin{remark}
As a byproduct, we also have the following estimates which is absent in \emph{Proposition \ref{comparison proposition}},
\begin{equation}\label{estimates for nabla_3 betab on H}
\|\nabla_3 \betab \|_{L^2_{(sc)}({H}_{u}^{(0,\ub)})}  \lesssim \delta^{-\frac{1}{2}}\Izero + C\delta^{\frac{1}{4}}.
\end{equation}
\end{remark}

\section{Energy Estimates on Outgoing Derivatives}
In this section, we derive energy estimates for $\nabla_4 \Psi$. We first take $N=L$ and all possible $X,Y,Z\in \{L,\Lb\}$,  except for the case $(X, Y, Z)=(L,L,L)$, in \eqref{basic energy identity} to derive\footnote{\quad We shall use letters $N_0$, $N_1$ and $N_2$ to denote either $L$ or $\Lb$ coming from $X$, $Y$ and $Z$.}
\begin{align*}
\int_{H_u} |\Psi(D_4 R)^{(s)}|^2 + \int_{{\Hb}_{\ub}}|\Psi(D_4 R)^{(s-\frac{1}{2})}|^2 &= \int_{H_0} |\Psi(D_4 R)^{(s)}|^2 \\
&~ + \doubleint_{\D(u,\ub)} D^\mu Q[D_4 R]_{\mu XYZ} + {}^{(N_0)}\!\pi_{\mu\nu}Q[D_4 R]^{\mu\nu}{}_{N_1 N_2}.
\end{align*}
We then multiply both sides by $\delta^{2s-3}$ and sum all those inequalities, in view of Remark \ref{remark forget delta}, we have
\begin{equation}\label{total energy for non anomalous L components}
\sum[ \|\Psi (D_4 R) \|_{L^2_{(sc)}(H_u^{(0,\ub)})}^2 + \|\Psi (D_4 R) \|_{L^2_{(sc)}({\Hb}_{\ub}^{(0,u)})}^2] = \Izero^2 +I +J + K
\end{equation}
where
\begin{equation*}
I + J=\doubleint_{\D(u,\ub)} D^\mu Q[D_4 R]_{\mu XYZ}, \quad K=\doubleint_{\D(u,\ub)}{}^{(N_0)}\!\pi_{\mu\nu}Q[D_4 R]^{\mu\nu}{}_{N_1 N_2}.
\end{equation*}

The sum in \eqref{total energy for non anomalous L components} is over on all possible $\Psi(D_4 R)$'s except $\alpha(D_4 R)$ and $\|\beta(D_4 R) \|_{L^2_{(sc)}({\Hb}_{\ub}^{(0,u)})}$. The definition for $I$, $J$ is in an obvious way as in the previous section.

\subsection{Estimates for $I$}
The integrand of $I$ can be written schematically as
\begin{equation*}
\Psi(D_4 R)^{(s_0)} \cdot \Psi^{(s_1)} \cdot \Psi^{(s_2)}
\end{equation*}
with total signature $3 \leq s_0 +s_1 +s_2 \leq 5$. We shall further split $I$ into
\begin{equation*}
I = I_0 + I_1 + I_2 + I_3
\end{equation*}
where $I_k$ denotes the collection of the terms whose integrand has exactly $k$ anomalies.

We observe that $I_3 =0 $ according to signature considerations. Otherwise, for a triple anomaly, thus $s_0 \geq 2$, $s_1 =2$ and $s_2 =2$ which contradicts the fact that $s_0 +s_1 +s_2 \leq 5$. We also notice that the only possibility for $\Psi(D_4 R)^{(s_0)}$ to be anomalous is $\Psi(D_4 R)^{(s_0)} = \alpha(D_4 R)$.

\subsubsection{Estimates for $I_0$}
The control of $I_0$ is much easier than other terms.
\begin{equation}\label{I'_0}
I_0 \lesssim \delta^{\frac{1}{2}} \int_0^{u} \|\Psi(D_4 R)_g^{(s_0)}\|_{L^2_{(sc)}({H}_{u'})}\|\Psi_g^{(s_1)}\|_{L^4_{(sc)}({H}_{u'})}\|\Psi_g^{(s_2)}\|_{L^4_{(sc)}({H}_{u'})}d u'\lesssim C \delta^{\frac{1}{2}}.
\end{equation}

\subsubsection{Estimates for $I_1$}
We further split $I_1$ into two terms
\begin{equation*}
I_1 = I_{11} + I_{12}
\end{equation*}
where the corresponding integrands for $I_{11}$ and $I_{12}$ are
\begin{equation*}
\Psi(D_4 R)^{(s_0)}_g \cdot \alpha \cdot \Psi^{(s_2)}_g ~~~\text{and}~~~ \alpha(D_4 R) \cdot \Psi^{(s_1)}_g \cdot \Psi^{(s_2)}_g
\end{equation*}
respectively.

$I_{11}$ can be bounded as follows,
\begin{equation*}
I_{11} \lesssim \delta^{\frac{1}{2}} \int_0^{u} \|\Psi(D_4 R)_g^{(s_0)}\|_{L^2_{(sc)}({H}_{u'})}\|\alpha\|_{L^4_{(sc)}({H}_{u'})}\|\Psi_g^{(s_2)}\|_{L^4_{(sc)}({H}_{u'})}d u'\lesssim C \delta^{\frac{1}{4}}.
\end{equation*}

$I_{12}$ requires an extra integration by parts. Before going into details, we first show how to replace $\alpha(D_4 R)$ by $\nabla_4 \alpha$. According to the proof of Lemma \ref{comparison lemma DN commute with Psi}, $\alpha(D_4 R)-\nabla_4 \alpha = \psi_g \cdot \Psi$. Thus,
\begin{equation*}
I_{12} = \doubleint_{\D(u,\ub)} \nabla_4 \alpha \cdot \Psi^{(s_1)}_g \cdot \Psi^{(s_2)}_g + \doubleint_{\D(u,\ub)}\psi_g \cdot \Psi \cdot \Psi^{(s_1)}_g \cdot \Psi^{(s_2)}_g =I_{121} + I_{122}
\end{equation*}
For $I_{122}$, since the nonlinearity gains one more term, we gain an extra $\delta^{\frac{1}{2}}$ thanks to scale invariant H\"older's inequality. Thus we can bound it as follows,
\begin{align*}
 I_{122}  &\lesssim \delta^{\frac{1}{2}}\cdot \delta^{\frac{1}{2}} \int_0^{u}\|\psi_g\|_{L^\infty_{(sc)}({H}_{u'})} \|\Psi\|_{L^2_{(sc)}({H}_{u'})}\|\Psi_g\|_{L^4_{(sc)}({H}_{u'})}\|\Psi_g^{(s_2)}\|_{L^4_{(sc)}({H}_{u'})} \lesssim C \delta^{\frac{1}{2}}.
\end{align*}
\begin{remark}\label{remark switch}
The above analysis allows us to freely switch between $\Psi(D_N R)$ and $\nabla_N \Psi$. The difference contributes at most $C \delta^{\frac{1}{4}}$ to energy estimates. In the rest of the paper, we shall ignore this $C \delta^{\frac{1}{4}}$ and we shall \textit{not} distinguish $\Psi(D_N R)$ from $\nabla_N \Psi$.
\end{remark}
We move on to $I_{12}$.
\begin{equation*}
I_{12} = I_{121} =\doubleint_{\D(u,\ub)} \nabla_4 \alpha \cdot \Psi_g \cdot \Psi_g = \doubleint_{\D(u,\ub)} \alpha \cdot \nabla_4 \Psi_g \cdot \Psi_g + \int_{\Hb_{\ub'}} \alpha \cdot \Psi_g \cdot \Psi_g \mid_{\ub'=0}^{\ub'=\ub}
\end{equation*}
For the bulk integral, in view of Lemma \ref{lemma estimates on nabla_N Psi}, we bound it as follows
\begin{equation*}
|\doubleint_{\D(u,\ub)} \alpha \cdot \nabla_4 \Psi_g \cdot \Psi_g| \lesssim \delta^{\frac{1}{2}} \int_0^{u} \|\nabla_4 \Psi_g\|_{L^2_{(sc)}({H}_{u'})}\|\alpha\|_{L^4_{(sc)}({H}_{u'})}\|\Psi_g\|_{L^4_{(sc)}({H}_{u'})} \lesssim C \delta^{\frac{1}{4}}.
\end{equation*}
For boundary integrals, we bound them as follows
\begin{equation*}
|\int_{\Hb_{\ub}} \alpha \cdot \Psi_g \cdot \Psi_g| \lesssim \delta^{\frac{1}{2}} \|\Psi_g\|_{L^2_{(sc)}({\Hb}_{\ub})}\|\alpha\|_{L^4_{(sc)}({\Hb}_{\ub})}\|\Psi_g\|_{L^4_{(sc)}({\Hb}_{\ub})} \lesssim C \delta^{\frac{1}{4}}.
\end{equation*}
Thus, we conclude $I_{12} \lesssim C \delta^{\frac{1}{4}}$. Combining the estimates for $I_{11}$ and $I_{12}$, we derive
\begin{equation}\label{I_1}
I_{1} \lesssim C \delta^{\frac{1}{4}}.
\end{equation}

\subsubsection{Estimates for $I_2$}
We split $I_2$ into two terms
\begin{equation*}
I_2 = I_{21} + I_{22}
\end{equation*}
where the corresponding integrands for $I_{21}$ and $I_{22}$ are
\begin{equation*}
\Psi(D_4 R)^{(s_0)}_g \cdot \alpha \cdot \alpha,  \quad \alpha(D_4 R) \cdot \alpha \cdot \Psi^{(s_2)}_g
\end{equation*}
respectively.

We first show that $I_{21} =0$. In view of \eqref{divergence of Q}, \eqref{divergence of D_N R}, \eqref{JN} and \eqref{JNstar}, the integrand of $I_{21}$ has the following expression
\begin{equation}\label{generic term}
D_N R_{X}{}^{\beta}{}_{Z}{}^{\delta} \cdot R^{\mu}{}_N{}_{Y}{}^{\nu} \cdot R_{\mu}{}_{\beta}{}_{\nu}{}_{\delta}.
\end{equation}
We write only one term in \eqref{generic term}. But this term is generic and the rest can be treated in the same way. If we look for $\alpha \cdot \alpha$, none of $\mu$ or $\nu$ can be 3 or 4; this forces $Y =  \beta = \delta=4$. Thus \eqref{generic term} reduces to
\begin{equation*}
D_4 R_{X}{}_{3}{}_{Z}{}_{3} \cdot \alpha \cdot \alpha.
\end{equation*}
If $X =\Lb$ or $Y=\Lb$, it vanishes; if $X=Y = L$, the the total signature is $7$, it is impossible. Hence $I_{21}$ must vanish.

For $I_{22}$, since total signature is at most $5$, thus $s_2 =0$. Thus the integrand of $I_{22}$ must be {}\footnote{\quad Since $\alpha$ and $\alphab$ are traceless, $\nabla_4 \alpha \cdot \alpha \cdot \alphab = \nabla_4 \alpha_{a}{}^{b}\,\alpha_{b}{}^{c} \,\alphab_{c}{}^{a}$.}
\begin{equation*}
\alpha(D_4 R) \cdot \alpha \cdot \alphab.
\end{equation*}
In view of Remark \ref{remark switch}, we have
\begin{align*}
I_{22} &=  \doubleint_{\D(u,\ub)} \nabla_4 \alpha \cdot \alpha \cdot \alphab = \frac{1}{2}\doubleint_{\D(u,\ub)} \nabla_4 (\alpha \cdot \alpha) \cdot \alphab= \doubleint_{\D(u,\ub)} \alpha \cdot \alpha \cdot \nabla_4 \alphab + \int_{\Hb_{\ub'}} \alpha \cdot \alpha \cdot \alphab \mid_{\ub'=0}^{\ub'=\ub}.
\end{align*}

For the bulk integral, we use Bianchi equation \eqref{NBE_L_alphab} to replace $\nabla_4 \alphab$. One can argue as in Remark \ref{remark switch}, only the top order term $\nabla \tensor \betab$ in \eqref{NBE_L_alphab} is relevant. The quadratic terms in \eqref{NBE_L_alphab} contribute at most a $C \delta^{\frac{1}{4}}$. By performing an extra integration by parts, we can move $\nabla$ to $\alpha$ to eliminate the anomaly of $\alpha$. More precisely, we proceed as follows	
\begin{align*}
 |\doubleint_{\D(u,\ub)} \alpha \cdot \alpha \cdot \nabla_4 \alphab| &\leq  |\doubleint_{\D(u,\ub)} \alpha \cdot \alpha \cdot \nabla \tensor \betab| + C \delta^{\frac{1}{4}} \lesssim |\doubleint_{\D(u,\ub)} \alpha \cdot \nabla \alpha \cdot \betab| + C \delta^{\frac{1}{4}}\\
 &\lesssim \delta^{\frac{1}{2}} \int_0^{u} \|\nabla \alpha\|_{L^2_{(sc)}({H}_{u'})}\|\alpha\|_{L^4_{(sc)}({H}_{u'})}\|\betab\|_{L^4_{(sc)}({H}_{u'})}+ C \delta^{\frac{1}{4}} \lesssim C \delta^{\frac{1}{4}}.
\end{align*}

To control the boundary integral, we need to improve Proposition \ref{L4 estimates for curvature},
\begin{lemma}\label{L4 alpha precise}
If $\delta$ is sufficiently small, then
\begin{equation*}
 \|\alpha\|_{L^4_{(sc)}(S)} \lesssim \delta^{-\frac{1}{4}} c(\Izero) + \delta^{-\frac{1}{4}} c(\Izero) \R^{\frac{1}{2}} + C\delta^{\frac{1}{16}}.
\end{equation*}
\end{lemma}
\begin{proof}
In view of Proposition \ref{Sobolev Trace}, we have
\begin{equation*}
 \|\alpha\|_{L^4_{(sc)}(S)} \lesssim (\delta^{\frac{1}{2}}\|\alpha\|_{L^2_{(sc)}(H_u)}+\|\nabla \alpha\|_{L^2_{(sc)}(H_u)})^{\frac{1}{2}}(\delta^{\frac{1}{2}}\|\alpha\|_{L^2_{(sc)}(H_u)}+\|\nabla_4 \alpha\|_{L^2_{(sc)}(H_u)})^{\frac{1}{2}},
\end{equation*}
Combined with Proposition \ref{energy estimates for curvature} and \eqref{estimates for nabla_4 alpha on H}, thus
\begin{equation}\label{eq1}
 \|\alpha\|_{L^4_{(sc)}(S)} \lesssim (\Izero + C \delta^{\frac{3}{4}}+\|\nabla \alpha\|_{L^2_{(sc)}(H_u)})^{\frac{1}{2}}((1+\delta^{-\frac{1}{2}})\Izero + C \delta^{\frac{1}{4}})^{\frac{1}{2}}.
\end{equation}
We then bound $\|\nabla \alpha\|_{L^2_{(sc)}(H_u)}$ by $\R$ to complete the proof.
\end{proof}
Thus, combined with Proposition \ref{energy estimates for curvature}, the boundary integral is controlled as follows
\begin{align*}
|\int_{\Hb_{\ub}} \alpha \cdot \alpha \cdot \alphab| &\leq \delta^{\frac{1}{2}} \|\alphab\|_{L^2_{(sc)}({\Hb}_{\ub})}\|\alpha\|_{L^4_{(sc)}({\Hb}_{\ub})}\|\alpha\|_{L^4_{(sc)}({\Hb}_{\ub})} \\
&\lesssim \delta^{\frac{1}{2}}(\Izero + c(\Izero)\R^{\frac{1}{2}} + C\delta^{\frac{1}{8}})(\delta^{-\frac{1}{4}} c(\Izero) + \delta^{-\frac{1}{4}} c(\Izero) \R^{\frac{1}{2}} + C\delta^{\frac{1}{16}})^2\\
&\lesssim c(\Izero) (1+ \R^{\frac{3}{2}}) +C\delta^{\frac{1}{8}}.
\end{align*}
Thus, we have the following estimates for $I_{22}$ and $I_{2}$,
\begin{equation*}
I_{2} = I_{22} \lesssim c(\Izero) (1+ \R^{\frac{3}{2}}) +C\delta^{\frac{1}{8}}.
\end{equation*}

\subsubsection{Conclusion}
Combining the estimates for $I_{1}$ and $I_{2}$, we derive
\begin{equation}\label{Estimate for L I}
I \lesssim c(\Izero) (1+ \R^{\frac{3}{2}}) +C\delta^{\frac{1}{8}}.
\end{equation}

\subsection{Estimates for $K$}
The integrand in $K$ can be written schematically as
\begin{equation*}
(\psi + \tr\chib_0)^{(s_0)} \cdot \Psi(D_4 R)^{(s_1)} \cdot \Psi(D_4 R)^{(s_2)}
\end{equation*}
with total signature $3 \leq s_0 +s_1 +s_2 \leq 5$. The appearance of $\tr\chib_0$ is due to ${}^{(\Lb)}\pi$. We split $K$ into
\begin{equation*}
K = K_0 + K_1 + K_2
\end{equation*}
where $K_k$ denotes the collection of terms with exactly $k$ $\alpha(D_4 R)$'s. We observe that $K_2 = 0$ according to the signature considerations.

\subsubsection{Estimates for $K_0$}
We first split $K_0$ into two terms
\begin{equation*}
K_2 = K_{01} + K_{02}
\end{equation*}
where the corresponding integrands for $K_{01}$ and $K_{02}$ are
\begin{equation*}
\psi^{(s_0)} \cdot \Psi(D_4 R)^{(s_1)}_g \cdot \Psi(D_4 R)^{(s_2)}_g ~~~\text{and}~~~ \tr\chib_0\cdot \Psi(D_4 R)^{(s_1)}_g \cdot \Psi(D_4 R)^{(s_2)}_g
\end{equation*}
respectively.

$K_{01}$ can be bounded as follows,
\begin{equation*}
K_{01} \leq \delta^{\frac{1}{2}} \int_0^{u} \|\psi\|_{L^{\infty}_{(sc)}({H}_{u'})}\|\Psi(D_4 R)^{(s_1)}_g\|_{L^2_{(sc)}({H}_{u'})}\|\Psi(D_4 R)^{(s_2)}_g\|_{L^2_{(sc)}({H}_{u'})}d u'\lesssim C \delta^{\frac{1}{2}}.
\end{equation*}

$K_{02}$ can be bounded in a similar way, although we lose a $\delta^{-\frac{1}{2}}$ due to $\tr\chib_0$,
\begin{align*}
K_{02} &\leq \int_0^{u} \|\Psi(D_4 R)^{(s_1)}_g\|_{L^2_{(sc)}({H}_{u'})}\|\Psi(D_4 R)^{(s_2)}_g\|_{L^2_{(sc)}({H}_{u'})}d \ub' \\
& \leq  \sum \int_0^{u} \|\Psi(D_4 R)_g\|_{L^2_{(sc)}({H}_{u'})}^2d \ub'
\end{align*}
Once we plug $K_{02}$ into \eqref{total energy for non anomalous L components}, in view of Gronwall's inequality, it can be absorbed by the left hand side of \eqref{total energy for non anomalous L components}. Hence we can ignore $K_{02}$ at this stage and simply write
\begin{equation}\label{K_0}
K_{0} \lesssim C \delta^{\frac{1}{2}}.
\end{equation}

\subsubsection{Estimates for $K_1$}\label{estimates for K_1}
First of all, we remark that $\tr\chib_0$ does not appear as connection coefficients in the integrand of $K_1$. Otherwise, in view of \eqref{divergence of P}, the appearance of $\tr\chib_0$ is through
\begin{equation*}
{}^{(\Lb)}\pi_{ab}\cdot Q(D_4 R)^{ab}{}_{N_1 N_2}.
\end{equation*}
Since $\alpha(D_4 R)$ appears, this forces $N_1 = N_2 =L$. Thus, the integrand can be written as
\begin{equation*}
{}^{(\Lb)}\pi_{ab}\cdot Q(D_4 R)^{ab}{}_{44} = \chib^{ab} (\rho \alpha_{ab} -2\sigma {}^*\!\alpha_{ab}).
\end{equation*}
Since $\alpha$ and ${}^*\!\alpha_{ab}$ are traceless, the $\tr\chib_0$ will be canceled. This argument also shows $\chih \cdot \alpha(D_4 R) \cdot \Psi(D_4 R)^{(s_2)}_g$ is absent in the integrand, otherwise they must be either $\chih \cdot \alpha(D_4 R) \cdot \rho(D_4 R)$ or $\chih \cdot \alpha(D_4 R) \cdot \sigma(D_4 R)$ with total signature $6$. This is impossible.

We further split $K_1$ into two terms
\begin{equation*}
K_1 = K_{11} + K_{12}
\end{equation*}
where the corresponding integrands for $K_{11}$ and $K_{12}$ are
\begin{equation*}
\chibh \cdot \alpha(D_4 R) \cdot \Psi(D_4 R)^{(s_2)}_g ~~~\text{and}~~~ \psi_g^{(s_0)} \cdot \alpha(D_4 R) \cdot \Psi(D_4 R)^{(s_2)}_g
\end{equation*}
respectively.

For $K_{11}$, as we just observed, $\Psi(D_4 R)^{(s_2)}_g = \rho(D_4 R)$ or $\Psi(D_4 R)^{(s_2)}_g = \sigma(D_4 R)$. We only treat the first possibility. The second one can be estimated exactly in the same way. In view of the Remark \ref{remark switch} and 	\eqref{NBE_L_rho}, we have
\begin{align*}
K_{11} &=  \doubleint_{\D(u,\ub)} \chibh \cdot \nabla_4 \alpha \cdot \nabla_4 \rho = \doubleint_{\D(u,\ub)} \chibh \cdot \nabla_4 \alpha  (-\frac{3}{2} \tr \chi \rho +\divergence \beta -\frac{1}{2}\chibh \cdot \alpha + \zeta \cdot \beta + 2\etab \cdot \beta)\\
&=  \doubleint_{\D(u,\ub)} \chibh \cdot \nabla_4 \alpha \cdot \nabla  \beta + \doubleint_{\D(u,\ub)} (\chibh \cdot \nabla_4 \alpha)(\chibh \cdot \alpha)+ C\delta^{\frac{1}{4}}=K_{111}+K_{112}+C\delta^{\frac{1}{4}}.
\end{align*}
For $K_{111}$, we reduce the number of anomalies by an extra integration by parts.
\begin{align*}
I_{111}&= \doubleint_{\D(u,\ub)}\nabla_4 \chibh \cdot  \alpha \cdot \nabla  \beta + \doubleint_{\D(u,\ub)}\chibh \cdot \alpha \cdot \nabla_4 \nabla  \beta+\int_{\Hb_{\ub'}} \chibh \cdot \alpha \cdot \nabla \beta \mid_{\ub'=0}^{\ub'=\ub}\\
&=K_{1111}+K_{1112}+K_{1113}.
\end{align*}

To estimate $K_{1111}$, we use null structure equation \eqref{NSE_L_chibh},
\begin{align*}
K_{1111} &= \doubleint_{\D(u,\ub)}( \nabla \tensor \etab +2\omega \chibh -\frac{1}{2}\tr \chib \chih +\etab \tensor \etab-\frac{1}{2} \tr \chi \chibh)\cdot  \alpha \cdot \nabla  \beta = \doubleint_{\D(u,\ub)} \chih \cdot  \alpha \cdot \nabla  \beta + C \delta^{\frac{1}{4}}\\
& = \doubleint_{\D(u,\ub)} \chih \cdot  \nabla \alpha \cdot \beta +\doubleint_{\D(u,\ub)} \nabla\chih \cdot  \alpha \cdot \beta+ C \delta^{\frac{1}{4}} \lesssim C \delta^{\frac{1}{4}}
\end{align*}	

To estimate $K_{1112}$, we compute the commutator
\begin{equation*}
 [\nabla_4,\nabla ] \beta = -\chi \cdot \nabla \beta + {}^*\!\beta \cdot \beta +\frac{1}{2}(\eta +\etab)\nabla_4 \beta + \etab \cdot \beta \cdot \chi,
\end{equation*}
We claim,
\begin{equation*}
| \doubleint_{\D(u,\ub)}\chibh \cdot \alpha \cdot [\nabla_4, \nabla]  \beta| \lesssim C \delta^{\frac{1}{4}}.
\end{equation*}
In fact, the worst term should be ${}^*\!\beta \cdot \beta$. We can use $L^\infty_{(sc)}$ norm on $\chibh$, $L^4_{(sc)}$ norm on two $\beta$'s and $L^2_{(sc)}$ on $\alpha$ to dominate it. Thus, up to an error of size $C\delta^{\frac{1}{4}}$
\begin{equation*}
K_{1112} =  \doubleint_{\D(u,\ub)} \chibh \cdot \alpha \cdot \nabla \nabla_4 \beta = \doubleint_{\D(u,\ub)} \nabla \chibh \cdot \alpha \cdot \nabla_4 \beta +   \doubleint_{\D(u,\ub)} \chibh \cdot \nabla \alpha \cdot \nabla_4 \beta.
\end{equation*}
For the first integral,
\begin{equation*}
 \doubleint_{\D(u,\ub)} \nabla \chibh \cdot \alpha \cdot \nabla_4 \beta  \lesssim \delta^{\frac{1}{2}} \int_0^{\ub} \|\nabla_4 \beta\|_{L^2_{(sc)}({H}_{u'})} \|\alpha\|_{L^4_{(sc)}({H}_{u'})}\|\nabla \chibh\|_{L^4_{(sc)}({H}_{u'})} \lesssim C \delta^{\frac{1}{4}}.
\end{equation*}
For the second integral,
\begin{equation*}
 \doubleint_{\D(u,\ub)} \chibh \cdot \nabla \alpha \cdot \nabla_4 \beta \lesssim \delta^{\frac{1}{2}} \int_0^{\ub} \|\nabla_4 \beta\|_{L^2_{(sc)}({H}_{u'})} \|\nabla \alpha\|_{L^2_{(sc)}({H}_{u'})}\|\chibh\|_{L^{\infty}_{(sc)}({H}_{u'})} \lesssim C \delta^{\frac{1}{4}}.
\end{equation*}
Hence,
\begin{equation*}
K_{1112} \lesssim C \delta^{\frac{1}{4}}.
\end{equation*}

To estimate $K_{1113}$, we proceed as follows,
\begin{align*}
K_{1113} &\leq|\int_{\Hb_{\ub}} \chibh \cdot \alpha \cdot \nabla \beta| \lesssim \delta^{\frac{1}{2}}  \|\nabla \beta\|_{L^2_{(sc)}({\Hb}_{\ub})} \|\alpha\|_{L^4_{(sc)}({\Hb}_{\ub})}\|\chibh\|_{L^4_{(sc)}({\Hb}_{\ub})}\\
&\lesssim  \delta^{\frac{1}{2}}\Rb (\delta^{-\frac{1}{4}} c(\Izero) + \delta^{-\frac{1}{4}}c(\Izero)\R^{\frac{1}{2}} + C\delta^{\frac{1}{16}})(\delta^{-\frac{1}{4}}+ C \delta^{\frac{1}{4}}) \lesssim c(\Izero) (\R+\Rb)^{\frac{3}{2}} +C\delta^{\frac{1}{4}}.
\end{align*}

Combining the estimates for $K_{1111}$, $K_{1112}$ and $K_{1113}$, we derive
\begin{equation*}
K_{111} \lesssim c(\Izero)  \R^{\frac{3}{2}} +C\delta^{\frac{1}{4}}.
\end{equation*}

We now turn to $K_{112}$.
\begin{align*}
 K_{112} &= \doubleint_{\D(u,\ub)} (\chibh \cdot \nabla_4 \alpha)(\chibh \cdot \alpha)= -\doubleint_{\D(u,\ub)} (\nabla_4 \chibh \cdot \alpha)(\chibh \cdot \alpha) + \frac{1}{2} \doubleint_{\D(u,\ub)} \nabla_4 |\chibh \cdot \alpha|^2\\
&=-\doubleint_{\D(u,\ub)} (\nabla_4 \chibh \cdot \alpha)(\chibh \cdot \alpha) + \frac{1}{2} \int_{{\Hb}_{\ub'}} \nabla_4 |\chibh \cdot \alpha|^2 \mid_{\ub' = 0}^{\ub'=\ub} \,= K_{1121} + K_{1122}.	
\end{align*}

To estimate $K_{1121}$, we use \eqref{NSE_L_chibh} to derive
\begin{align*}
 K_{1121}&= \doubleint_{\D(u,\ub)} ( \nabla \tensor \etab +2\omega \chibh -\frac{1}{2}\tr \chib \chih +\etab \tensor \etab-\frac{1}{2} \tr \chi \chibh)\cdot \alpha \cdot \chibh \cdot \alpha =\doubleint_{\D(u,\ub)} \chih 	\cdot \alpha \cdot \chibh \cdot \alpha + C\delta^{\frac{1}{4}}\\
&\lesssim \delta^{\frac{3}{2}}\int_{0}^{u} \|\chih\|_{L^4_{(sc)}({H}_{u})}\|\alpha\|_{L^4_{(sc)}({H}_{u})}\|\chibh\|_{L^4_{(sc)}({H}_{u})}\|\alpha\|_{L^4_{(sc)}({H}_{u})}+C\delta^{\frac{1}{4}} \lesssim C\delta^{\frac{1}{4}}.
\end{align*}
To estimate $K_{1122}$, we proceed as follows
\begin{align*}
 K_{1122} \lesssim \delta^{\frac{3}{2}}\|\chibh\|^2_{L^4_{(sc)}({H}_{u})}\|\alpha\|^2_{L^4_{(sc)}({H}_{u})}\lesssim C\delta^{\frac{1}{4}}.
\end{align*}
Thus,  we have $ K_{112} \lesssim  C\delta^{\frac{1}{4}}$. Hence,
\begin{equation*}
K_{11} \lesssim c(\Izero)  (\R+\Rb)^{\frac{3}{2}} +C\delta^{\frac{1}{4}}.
\end{equation*}

It remains to control $K_{12}$ whose integrand is $\psi_g^{(s_0)} \cdot \alpha(D_4 R) \cdot \Psi(D_4 R)^{(s_2)}_g$. The key \footnote{\quad In \cite{K-R-09}, the authors derived all $L^4_{(sc)}$ estimates for $\nabla_4 \psi$ except for the case $\psi = \omega$.} is to observe $\psi_g \neq \omega$. Otherwise, since total signature is at most $5$, it forces the integrand to be $\omega \cdot \alpha(D_4 R) \cdot \alphab(D_4 R)$. Once again, we investigate the origin of this term
\begin{equation*}
{}^{(\Lb)}\pi_{\mu\nu}\cdot Q(D_4 R)^{\mu\nu}{}_{N_1 N_2}.
\end{equation*}
By direction computation, we see that $\alpha(D_4 R) \cdot \alphab(D_4 R)$ never appears in $Q(D_4 R)^{\mu\nu}{}_{N_1 N_2}$, hence we arrive at a contradiction.

Based on the above observation, since we have $L^{4}_{(sc)}$ estimates on $\nabla_4 \psi_g$, we can proceed exactly as we did for $K_{11}$ to derive
\begin{equation*}
K_{12} \lesssim C\delta^{\frac{1}{4}}.
\end{equation*}
In this case, since $\psi_g$ is not anomalous, the estimates are much easier to derive. We skip details.
Combining the estimates for $K_{11}$ and $K_{12}$, we derive
\begin{equation}\label{K_1}
K_{1} \lesssim c(\Izero) (\R+\Rb)^{\frac{3}{2}} +C\delta^{\frac{1}{4}}.
\end{equation}

\subsubsection{Conclusion}
Combining the estimates for $K_0$ and $K_1$, we derive
\begin{equation}\label{Estimate for L K}
K \lesssim c(\Izero)(\R+\Rb)^{\frac{3}{2}} +C\delta^{\frac{1}{4}}.
\end{equation}

\subsection{Estimates for $J$}
The integrand of $J$ can be written schematically as
\begin{equation*}
\psi^{(s_0)} \cdot \Psi(D_4 R)^{(s_1)} \cdot \Psi(D_\nu R)^{(s_2)}
\end{equation*}
with total signature $3 \leq s_0 +s_1 +s_2 \leq 5$.  We split $J$ into
\begin{equation*}
J = J_0 + J_1 + J_2
\end{equation*}
where $J_k$ denotes the collection of terms with exactly $k$ anomalous curvature components.

We make the following observation: the only anomaly for $\Psi(D_4 R)^{(s_1)}$ is $\alpha(D_4 R)$; things are different for $\Psi(D_\nu R)^{(s_2)}$, it has multiple choices:  $\alpha(D_4 R)$, $\alphab(D_3 R)$, $\alpha(D_3 R)$ and $\beta(D_c R)$.  We also recall Section \ref{mild anomaly} which asserts that last two anomalies are mild in the sense that they all come from zero order term $\alpha$. This will be extremely crucial in our proof.

\subsubsection{Estimates for $J_0$}	
The integrand of $J_0$, by definition, has the following schematic expressions
\begin{equation*}
\psi^{(s_0)} \cdot \Psi(D_4 R)^{(s_1)}_g \cdot \Psi(D_\nu R)^{(s_2)}_g.
\end{equation*}
Thus, we bound $J_0$ as follows,
\begin{equation*}
J_0 \leq \delta^{\frac{1}{2}} \int_0^{\ub} \|\psi\|_{L^{\infty}_{(sc)}({H}_{u'})}\|\Psi(D_4 R)^{(s_1)}_g\|_{L^2_{(sc)}({H}_{u'})}\|\Psi(D_\nu R)^{(s_2)}_g\|_{L^2_{(sc)}({H}_{u'})}d \ub'\lesssim C \delta^{\frac{1}{2}}.
\end{equation*}

\subsubsection{Estimates for $J_1$}
We split $J_1$ into $J_1 = J_{11} + J_{12}$. For $J_{11}$, $\Psi(D_4 R)^{(s_1)} = \alpha(D_4 R)$; for $J_{12}$, $\Psi(D_4 R)^{(s_1)} = \Psi(D_4 R)^{(s_1)}_g$ is not anomalous.

The integrand of $J_{11}$ has the following form (recall that we ignore the Hodge dual part)
\begin{equation*}
 D_4 R_{X}{}^{\rho}{}_{Y}{}^{\delta} \cdot D^{\mu} L^{\nu} \cdot D_\nu R_{\mu \rho Z \delta}.
\end{equation*}
Since $\Psi(D_4 R)^{(s_1)} = \alpha(D_4 R)$, it forces $X=Y=L$, $Z= \Lb$, $\rho = a$ and $\delta =b$. Thus, the integrand is reduced to
\begin{equation*}
 \nabla_4 \alpha \cdot D^{\mu} L^{\nu} \cdot D_\nu R_{\mu a 3 b}.
\end{equation*}
Since $D^{\mu} L^{3} = 0$, we see $\nu \neq 3$. If $\nu = 4$, we use Bianchi identities, up to lower order terms, to convert $ D_\nu R_{\mu a 3 b}$ into the form $\nabla \Psi_g$. Hence, the only possible forms of the integrand are either $\chih \cdot \nabla_4 \alpha \cdot \nabla \Psi_g$ or $\psi_g \cdot \nabla_4 \alpha \cdot \nabla \Psi_g$ ($\psi_g \neq \omega$). We then split $J_{11}$ as
\begin{equation*}
 J_{11} = \doubleint_{\D(u,\ub)} \chih \cdot \nabla_4 \alpha \cdot \nabla \Psi_g + \doubleint_{\D(u,\ub)} \psi_g \cdot \nabla_4 \alpha \cdot \nabla \Psi_g = J_{111}+J_{112}.
\end{equation*}
Notice that $J_{112}$ has exactly the same form as $K_{12}$ term in the previous section, thus
\begin{equation*}
 J_{112} \lesssim C \delta^{\frac{1}{2}}.
\end{equation*}
The estimates for $J_{111}$ is quite similar to $K_{11}$. First of all, according to signature consideration, $\nabla \Psi_g =\nabla \betab$, thus
\begin{align*}
 J_{111} &= \doubleint_{\D(u,\ub)}\nabla_4 \chih \cdot  \alpha \cdot \nabla  \betab + \doubleint_{\D(u,\ub)}\chih \cdot \alpha \cdot \nabla_4 \nabla  \betab+\int_{\Hb_{\ub'}} \chih \cdot \alpha \cdot \nabla \betab \mid_{\ub'=0}^{\ub'=\ub}\\
&=J_{1111}+J_{1112}+J_{1113}.	
\end{align*}

To estimate $J_{111}$, we use null structure equation \eqref{NSE_L_chi},
\begin{align*}
J_{1111} &= \doubleint_{\D(u,\ub)}\nabla_4 \chih \cdot  \alpha \cdot \nabla  \betab = \doubleint_{\D(u,\ub)}(-\tr \chi \,\chih  -2\omega \chih -\alpha)\cdot  \alpha \cdot \nabla  \betab\\
&= \doubleint_{\D(u,\ub)} \alpha \cdot  \alpha \cdot \nabla  \betab + C \delta^{\frac{1}{4}} = \doubleint_{\D(u,\ub)} \alpha \cdot  \nabla \alpha \cdot \betab + C \delta^{\frac{1}{4}}.
\end{align*}
Hence,
\begin{equation*}
J_{1111} \lesssim \delta^{\frac{1}{2}} \int_0^{\ub} \|\nabla \alpha\|_{L^2_{(sc)}({H}_{u'})}\|\alpha\|_{L^4_{(sc)}({H}_{u'})}\|\beta\|_{L^4_{(sc)}({H}_{u'})} + C \delta^{\frac{1}{4}} \lesssim C \delta^{\frac{1}{4}}.
\end{equation*}

To estimate $J_{1112}$, we compute commutator
\begin{equation*}
 [\nabla_4,\nabla ] \betab = -\chi \cdot \nabla \betab + {}^*\!\beta \cdot \betab +\frac{1}{2}(\eta +\etab)\nabla_4 \betab + \etab \cdot \betab \cdot \chi,
\end{equation*}
Using this formula, we can easily derive
\begin{equation*}
| \doubleint_{\D(u,\ub)}\chih \cdot \alpha \cdot [\nabla_4, \nabla]  \betab| \lesssim C \delta^{\frac{1}{4}}.
\end{equation*}
Hence, up to an error of size $C \delta^{\frac{1}{4}}$, we have
\begin{align*}
J_{1112} =  \doubleint_{\D(u,\ub)} \chih \cdot \alpha \cdot \nabla \nabla_4 \betab= \doubleint_{\D(u,\ub)} \nabla \chih \cdot \alpha \cdot \nabla_4 \betab +   \doubleint_{\D(u,\ub)} \chih \cdot \nabla \alpha \cdot \nabla_4 \betab \lesssim C\delta^{\frac{1}{4}}.
\end{align*}
To estimate $J_{1113}$, we proceed as follows,
\begin{align*}
J_{1113} &\leq|\int_{\Hb_{\ub}} \chih \cdot \alpha \cdot \nabla \betab| \lesssim \delta^{\frac{1}{2}}  \|\nabla \betab\|_{L^2_{(sc)}({\Hb}_{\ub})} \|\alpha\|_{L^4_{(sc)}({\Hb}_{\ub})}\|\chih\|_{L^4_{(sc)}({\Hb}_{\ub})}\\
&\lesssim  \delta^{\frac{1}{2}}\Rb (\delta^{-\frac{1}{4}} c(\Izero) + \delta^{-\frac{1}{4}}c(\Izero)\R^{\frac{1}{2}} + C\delta^{\frac{1}{16}})(\delta^{-\frac{1}{4}}+ C \delta^{\frac{1}{4}}) \lesssim c(\Izero) (\R+\Rb)^{\frac{3}{2}} +C\delta^{\frac{1}{4}}.
\end{align*}
Putting things together, we have
\begin{equation*}
J_{11} \lesssim c(\Izero) (\R+\Rb)^{\frac{3}{2}} +C\delta^{\frac{1}{4}}.
\end{equation*}

We move on to the estimates for $J_{12}$ where $\Psi(D_\nu R)^{(s_2)}$ must be anomalous. Thus, because $D^\mu L^{3} =0 $, it must be either $\alpha(D_4 R)$ or $\beta(D_a R)$. We can assume $\Psi(D_\nu R)^{(s_2)} \neq \alpha(D_4 R)$. Otherwise, by signature considerations, those terms (of the form $\nabla \Psi_g \cdot \psi \cdot \nabla_4 \alpha$) have already been treated in $J_{11}$, so we ignore them at this stage. The anomalous term must be $\beta(D_a R)$. As we emphasized at the beginning of the section, its anomaly comes from the lower derivative term $\alpha$. Hence, up to an error of size $C\delta^{\frac{1}{4}}$, we can replace the integrand by
\begin{equation*}
 \nabla \Psi_g \cdot \psi \cdot \alpha +  \nabla \Psi \cdot \psi_g \cdot \alpha
\end{equation*}
where  $ \nabla \Psi_g \neq \nabla \alpha$ according to signature considerations. The second term can be estimated easily by placing $L^{4}_{(sc)}$ estimates on $\psi_g$ and $\alpha$ (which saves an extra $\delta^{\frac{1}{4}}$), thus we bound $J_{12}$ as follows
\begin{align*}
J_{12} &=C\delta^{\frac{1}{4}}+ \doubleint_{\D(u,\ub)} \psi \cdot \alpha \cdot \nabla \Psi_g = C\delta^{\frac{1}{4}}+\doubleint_{\D(u,\ub)} \nabla \psi \cdot \alpha \cdot \Psi_g+ \doubleint_{\D(u,\ub)} \psi \cdot \nabla \alpha \cdot \Psi_g\\
& \lesssim C \delta^{\frac{1}{4}}+\delta^{\frac{1}{2}} \int_0^{u} \|\nabla \psi\|_{L^2_{(sc)}} \|\alpha\|_{L^4_{(sc)}}\|\Psi_g\|_{L^4_{(sc)}} +  \|\psi\|_{L^\infty_{(sc)}} \|\nabla \alpha\|_{L^2_{(sc)}}\|\Psi_g\|_{L^2_{(sc)}} \lesssim C \delta^{\frac{1}{4}}.
\end{align*}
Finally, we can bound $J_1$ by
\begin{equation*}
J_{1} \lesssim (\R+\Rb)^{\frac{3}{2}} +C\delta^{\frac{1}{4}}.
\end{equation*}

\subsubsection{Vanishing for $J_2$}
We claim $J_2 =0$. Since both curvature terms in the integrand
\begin{equation*}
 D_4 R_{X}{}^{\rho}{}_{Y}{}^{\delta} \cdot D^{\mu} N^{\nu} \cdot D_\nu R_{\mu \rho Z \delta},
\end{equation*}
are anomalous, it must be
\begin{equation*}
 D_4 R_{4}{}^{a}{}_{4}{}^{b} \cdot D^{\mu} L^{\nu} \cdot D_\nu R_{\mu a 3 b},
\end{equation*}
Since total signature is $5$ at this situation, the second anomaly must be either $\alpha(D_3 R)$ or $\beta(D_c R)$. Their signature must be $2$. Once again, since $D^{\mu}L^3 =0$, thus the second anomaly must be $\beta(D_c R)$. Thus, $\nu = c$. Finally, since the signature of $D_c R_{\mu a 3 b}$ is at most $\frac{3}{2}$, it can not be anomalous. This shows $J_2 =0$.



\subsubsection{Conclusion}
Combining the estimates for $J_0$ and $J_1$, we finally derive estimates for $J$ as follows,
\begin{equation}\label{Estimate for L J}
J \lesssim c(\Izero) (\R + \Rb)^{\frac{3}{2}} +C\delta^{\frac{1}{4}}.
\end{equation}

We combine \eqref{total energy for non anomalous L components}, \eqref{Estimate for L I}, \eqref{Estimate for L K} and \eqref{Estimate for L J}, we conclude that
\begin{equation}\label{Energy Estimates part A}
\sum[ \|\Psi (D_4 R) \|_{L^2_{(sc)}(H_u^{(0,\ub)})}^2 + \|\Psi (D_4 R) \|_{L^2_{(sc)}({\Hb}_{\ub}^{(0,u)})}^2] \lesssim c(\Izero) (1+ (\R + \Rb)^{\frac{3}{2}}) +C\delta^{\frac{1}{8}}.
\end{equation}
Considering $\|\betab(D_4 R) \|_{L^2_{(sc)}({\Hb}_{\ub}^{(0,u)})}$ in \eqref{Energy Estimates part A}, we have
\begin{equation*}
 \|\nabla_4 \betab \|_{L^2_{(sc)}({\Hb}_{\ub}^{(0,u)})} \lesssim c(\Izero) (1+ (\R + \Rb)^{\frac{3}{2}}) +C\delta^{\frac{1}{8}}.
\end{equation*}
In view of\eqref{NBE_L_betab}, modulo quadratic terms which gives an error of size $C\delta^{\frac{1}{4}}$, we have
\begin{equation*}
\|-\nabla \rho + ^*\! \nabla \sigma\|_{L^2_{(sc)}({\Hb}_{\ub}^{(0,u)})} \lesssim c(\Izero) (1+ (\R + \Rb)^{\frac{3}{2}}) +C\delta^{\frac{1}{8}}.
\end{equation*}
According to the standard elliptic estimates for Hodge systems, we derive
\begin{equation*}\label{nabla rho sigma on Hb}
\|(\nabla \rho,  \nabla \sigma)\|_{L^2_{(sc)}({\Hb}_{\ub}^{(0,u)})} \lesssim c(\Izero) (1+ (\R + \Rb)^{\frac{3}{2}}) +C\delta^{\frac{1}{8}}.
\end{equation*}
Considering $\|\betab (D_4 R) \|_{L^2_{(sc)}({H}_{u}^{(0,u)})}$ in \eqref{Energy Estimates part A}, similarly, we derive
\begin{equation*}\label{nabla rho sigma on H}
\|(\nabla \rho,  \nabla \sigma)\|_{L^2_{(sc)}({H}_{u}^{(0,u)})} \lesssim c(\Izero) (1+ (\R + \Rb)^{\frac{3}{2}}) +C\delta^{\frac{1}{8}}.
\end{equation*}
Considering $\|\beta (D_4 R) \|_{L^2_{(sc)}({H}_{u}^{(0,u)})}$ in \eqref{Energy Estimates part A} and \eqref{NBE_L_beta}, we derive
\begin{equation*}\label{nabla alpha on H}
\|\nabla \alpha\|_{L^2_{(sc)}({H}_{u}^{(0,u)})} \lesssim c(\Izero) (1+ (\R + \Rb)^{\frac{3}{2}}) +C\delta^{\frac{1}{8}}.
\end{equation*}
Considering $\|\rho (D_4 R) \|_{L^2_{(sc)}({H}_{u}^{(0,u)})}$, $\|\sigma (D_4 R) \|_{L^2_{(sc)}({H}_{u}^{(0,u)})}$ in \eqref{Energy Estimates part A} and \eqref{NBE_L_rho} as well as \eqref{NBE_L_sigma}, we derive
\begin{equation*}\label{nabla beta on H}
\|\nabla \beta\|_{L^2_{(sc)}({H}_{u}^{(0,u)})} \lesssim c(\Izero) (1+ (\R + \Rb)^{\frac{3}{2}}) +C\delta^{\frac{1}{8}}.
\end{equation*}
Similarly, we have
\begin{equation*}\label{nabla beta on Hb}
\|\nabla \beta\|_{L^2_{(sc)}({\Hb}_{\ub}^{(0,u)})} \lesssim c(\Izero) (1+ (\R + \Rb)^{\frac{3}{2}}) +C\delta^{\frac{1}{8}}.
\end{equation*}
Putting all them all together, we have
\begin{align}\label{Close Bootstrap 1}
\|(\nabla \alpha, \nabla \beta, \nabla \rho, \nabla \sigma)\|_{L^2_{(sc)}({H}_{u})}  + \|(\nabla \beta, \nabla \rho,  & \nabla \sigma)\|_{L^2_{(sc)}({\Hb}_{\ub})} \notag \\
\quad & \lesssim c(\Izero) (1+ (\R + \Rb)^{\frac{3}{2}}) +C\delta^{\frac{1}{8}}.
\end{align}

\section{Energy Estimates on Incoming Derivatives}\label{Energy Estimates on Incoming Directions}
We take $N=\Lb$, $(X, Y, Z)=(L,L,\Lb)$ or $(X, Y, Z)=(L,\Lb,\Lb)$, in \eqref{basic energy identity} to derive
\begin{align*}
\int_{H_u} |\Psi(D_3 R)^{(s)}|^2 + \int_{{\Hb}_{\ub}}&|\Psi(D_3 R)^{(s-\frac{1}{2})}|^2 \lesssim \int_{H_0} |\Psi(D_3 R)^{(s)}|^2 \\
&\quad + \doubleint_{\D(u,\ub)}| D^\mu Q[D_3 R]_{\mu XYZ} + {}^{(N_0)}\!\pi_{\mu\nu}Q[D_3 R]^{\mu\nu}{}_{N_1 N_2}|.
\end{align*}
We multiply both sides by $\delta^{2s-3}$ and sum those two inequalities, in view of Remark \ref{remark forget delta}, we have
\begin{align}\label{total energy for non anomalous Lb components}
&\|(\beta(D_3 R), \rho(D_3 R), \sigma(D_3 R))\|_{L^2_{(sc)}(H_u^{(0,\ub)})}^2 + \|(\rho(D_3 R), \sigma(D_3 R), \betab(D_3 R)) \|_{L^2_{(sc)}({\Hb}_{\ub}^{(0,u)})}^2 \\
&\quad \lesssim \Izero^2 +I + J+ K\notag
\end{align}
where
\begin{equation*}
I + J=|\doubleint_{\D(u,\ub)} D^\mu Q[D_3 R]_{\mu XYZ}|, \quad K=|\doubleint_{\D(u,\ub)}{}^{(N_0)}\!\pi_{\mu\nu}Q[D_3 R]^{\mu\nu}{}_{N_1 N_2}|.
\end{equation*}

In this section, the possible anomalies for $\Psi(D_3 R)^{(s_1)}$ are $\alphab(D_3 R)$ and $\alpha(D_3 R)$; the possible anomalies for $\Psi(D_\nu R)^{(s_2)}$ are  $\alpha(D_4 R)$, $\alphab(D_3 R)$, $\alpha(D_3 R)$ and $\beta(D_c R)$.  We still emphasize that the anomalies for $\alpha(D_3 R)$ and $\beta(D_c R)$ are mild in nature. They come from the lower derivative anomaly $\alpha$. For example, in applications, we always use the following expressions,
\begin{equation}\label{anomalous behavior for alpha(D_3 R)}
\alpha(D_3 R) = \alpha + C \delta^{\frac{1}{4}}, ~\beta(D_a R) = \alpha + C \delta^{\frac{1}{4}}.
\end{equation}

\subsection{Estimates for $I$}
The integrand of $I$ can be written schematically as
\begin{equation*}
\Psi(D_3 R)^{(s_0)} \cdot \Psi^{(s_1)} \cdot \Psi^{(s_2)}
\end{equation*}
with total signature $ s_0 +s_1 +s_2 = 2$ or $3$. We split $I$ into
\begin{equation*}
I = I_0 + I_1 + I_2 + I_3
\end{equation*}
where $I_k$ denotes the collection of terms with exactly $k$ anomalies. According to the signature considerations, $I_3 = 0$, otherwise $s_2+s_3 \geq 4$.

\subsubsection{Estimates for $I_0$}
We control $I_0$ as follows,
\begin{equation}\label{I_0}
I_0 \lesssim \delta^{\frac{1}{2}}\delta^{-1} \int_0^{\ub} \|\Psi(D_3 R)_g^{(s_0)}\|_{L^2_{(sc)}({\Hb}_{\ub'})}\|\Psi_g^{(s_1)}\|_{L^4_{(sc)}({\Hb}_{\ub'})}\|\Psi_g^{(s_2)}\|_{L^4_{(sc)}({\Hb}_{\ub'})}d \ub' \lesssim C \delta^{\frac{1}{2}}.
\end{equation}

\subsubsection{Estimates for $I_1$}
We split $I_1$ into three terms
\begin{equation*}
I_1 = I_{11} + I_{12} + I_{13}
\end{equation*}
where the corresponding integrands for $I_{11}$, $I_{12}$ and $I_{13}$ are
\begin{equation*}
\Psi(D_3 R)^{(s_0)}_g \cdot \alpha \cdot \Psi^{(s_2)}_g, \quad \alpha(D_3 R) \cdot \Psi^{(s_1)}_g \cdot \Psi^{(s_2)}_g, \quad \alphab(D_3 R) \cdot \Psi^{(s_1)}_g \cdot \Psi^{(s_2)}_g
\end{equation*}
respectively.

To estimate $I_{11}$, we have
\begin{equation*}
I_{11} \lesssim \delta^{\frac{1}{2}} \delta^{-1} \int_0^{\ub} \|\Psi(D_3 R)_g^{(s_0)}\|_{L^2_{(sc)}({\Hb}_{\ub'})}\|\alpha\|_{L^4_{(sc)}({\Hb}_{\ub'})}\|\Psi_g^{(s_2)}\|_{L^4_{(sc)}({\Hb}_{\ub'})}d \ub'\lesssim C \delta^{\frac{1}{4}}.
\end{equation*}

To estimate $I_{12}$, in view of \eqref{anomalous behavior for alpha(D_3 R)}, modulo $C\delta^{\frac{1}{4}}$, we have
\begin{align*}
I_{12} &= \doubleint_{\D(u,\ub)} \alpha \cdot \Psi_g \cdot \Psi_g  \lesssim \delta^{\frac{1}{2}} \delta^{-1} \int_0^{\ub} \|\alpha\|_{L^4_{(sc)}({\Hb}_{\ub'})}\|\Psi_g\|_{L^2_{(sc)}({\Hb}_{\ub'})}\|\Psi_g\|_{L^4_{(sc)}({\Hb}_{\ub'})}\lesssim C \delta^{\frac{1}{4}}.
\end{align*}

To estimate $I_{13}$, we observe that none of the $\Psi_g$'s is $\alphab$. Otherwise, since the total signature is at least $2$, this force another $\Psi_g$ to be $\alpha$ which is impossible(in this case, we only assume $I_1$ contains one anomaly). Thus, up to an error of size $C\delta^{\frac{1}{4}}$,
\begin{equation*}
I_{13} =  \doubleint_{\D(u,\ub)} \nabla_3 \alphab \cdot \Psi_g \cdot \Psi_g = \doubleint_{\D(u,\ub)} \alphab \cdot \nabla_3 \Psi_g \cdot \Psi_g + \int_{H_{u'}} \alphab \cdot \Psi_g \cdot \Psi_g \mid_{u'=0}^{u'=u}.
\end{equation*}
For the bulk integral, we bound it as follows
\begin{equation*}
|\doubleint_{\D(u,\ub)} \alphab \cdot \nabla_3 \Psi_g \cdot \Psi_g| \lesssim \delta^{\frac{1}{2}} \int_0^{u} \|\nabla_3 \Psi_g\|_{L^2_{(sc)}({H}_{u'})}\|\alphab\|_{L^4_{(sc)}({H}_{u'})}\|\Psi_g\|_{L^4_{(sc)}({H}_{u'})} \lesssim C \delta^{\frac{1}{2}}.
\end{equation*}
For boundary integrals, we bound them as follows
\begin{equation*}
|\int_{H_{u}} \alphab \cdot \Psi_g \cdot \Psi_g| \lesssim \delta^{\frac{1}{2}}\|\alphab\|_{L^2_{(sc)}({H}_{u})} \|\Psi_g\|_{L^4_{(sc)}({H}_{u})}\|\Psi_g\|_{L^4_{(sc)}({H}_{u})} \lesssim C \delta^{\frac{1}{2}}.
\end{equation*}
Hence, we bound $I_{13} \lesssim C \delta^{\frac{1}{4}}$. Combined with the estimates for $I_{11}$ and $I_{12}$, we derive
\begin{equation}\label{I_1_Lb}
I_{1} \lesssim C \delta^{\frac{1}{4}}.
\end{equation}

\subsubsection{Estimates for $I_2$}
According to signature considerations, the integrand of $I_2$ must have the following schematic form
\begin{equation*}
\alphab(D_3 R) \cdot \alpha \cdot \Psi^{(s_2)}_g.
\end{equation*}

Since the total signature is $2$ or $3$, thus $\Psi_g \in \{\alphab, \rho, \sigma\}$. The estimates for $\Psi_g = \rho$ and $\Psi_g = \sigma$ are similar, we shall only concentrate on $\Psi_g  = \rho$. Hence, we further split $I_{2}$ into
\begin{equation*}
I_{2} = I_{21} + I_{22} = \doubleint_{\D(u,\ub)}\nabla_3 \alphab \cdot \alpha \cdot \alphab +  \doubleint_{\D(u,\ub)}\nabla_3 \alphab \cdot \alpha \cdot \rho.
\end{equation*}
To control $I_{21}$, we use move the derivative to a better position,
\begin{equation*}
I_{21} =  \doubleint_{\D(u,\ub)} \nabla_3 (\alphab \cdot \alphab) \cdot \alpha = \doubleint_{\D(u,\ub)} \alphab \cdot \alphab \cdot \nabla_3 \alpha+ \int_{H_{u'}}\alphab \cdot \alphab \cdot \alpha \mid_{u'=0}^{u'=u}.
\end{equation*}
In view of \eqref{anomalous behavior for alpha(D_3 R)}, up to an error of size $C\delta^{\frac{1}{4}}$, the bulk integral is
\begin{equation*}
\doubleint_{\D(u,\ub)} \alphab \cdot \alphab \cdot  \alpha
\end{equation*}
Then we can place $L^4_{(sc)}$ norm on $\alpha$ to save an $\delta^{\frac{1}{4}}$ to derive
\begin{equation*}
I_{21} \lesssim |\doubleint_{\D(u,\ub)} \alphab \cdot \alphab \cdot  \alpha| +  |\int_{H_{u}}\alphab \cdot \alphab \cdot \alpha| \lesssim C \delta^{\frac{1}{4}}.
\end{equation*}
To control $I_{22}$, we still integrate by parts,
\begin{equation*}
I_{22} =  \doubleint_{\D(u,\ub)} \nabla_3 \alphab  \cdot \alpha \cdot \rho= \doubleint_{\D(u,\ub)} \alphab \cdot \rho \cdot \nabla_3 \alpha+\doubleint_{\D(u,\ub)} \alphab \cdot \nabla_3 \rho \cdot \alpha+ \int_{H_{u'}}\alphab \cdot \rho \cdot \alpha \mid_{u'=0}^{u'=u}.
\end{equation*}
Exactly as for $I_{21}$, we can bound each term by $C \delta^{\frac{1}{4}}$ easily. Thus,
\begin{equation}\label{I_2_Lb}
I_{2} \lesssim I_{21} + I_{22} \lesssim C \delta^{\frac{1}{4}}.
\end{equation}

\subsubsection{Conclusion}
Combining the estimates for $I_{1}$ and $I_{2}$, we derive
\begin{equation}\label{Estimate for Lb I}
I \lesssim C \delta^{\frac{1}{4}}.
\end{equation}

\subsection{Estimates for $K$}
The integrand of $K$ can be written schematically as
\begin{equation*}
(\psi + \tr\chib_0)^{(s_0)} \cdot \Psi(D_3 R)^{(s_1)} \cdot \Psi(D_3 R)^{(s_2)}
\end{equation*}
with total signature $s_0 +s_1 +s_2 = 2$ or $3$. We split $K$ into
\begin{equation*}
K = K_0 + K_1 + K_2
\end{equation*}
where $K_k$ denotes the collection of terms with exactly $k$ anomalous curvature components.

\subsubsection{Estimates for $K_0$}
We split $K_0$ into two terms
\begin{equation*}
K_2 = K_{01} + K_{02}
\end{equation*}
where the corresponding integrands for $K_{01}$ and $K_{02}$ are
\begin{equation*}
\psi^{(s_0)} \cdot \Psi(D_3 R)^{(s_1)}_g \cdot \Psi(D_3 R)^{(s_2)}_g, \quad \tr\chib_0\cdot \Psi(D_3 R)^{(s_1)}_g \cdot \Psi(D_3 R)^{(s_2)}_g
\end{equation*}
respectively.

$K_{01}$ can be bounded as follows,
\begin{equation*}
K_{01} \leq \delta^{\frac{1}{2}} \int_0^{u} \|\psi\|_{L^{\infty}_{(sc)}({H}_{u'})}\|\Psi(D_3 R)^{(s_1)}_g\|_{L^2_{(sc)}({H}_{u'})}\|\Psi(D_3 R)^{(s_2)}_g\|_{L^2_{(sc)}({H}_{u'})}d u'\lesssim C \delta^{\frac{1}{2}}.
\end{equation*}

$K_{02}$ can be bounded as follows,
\begin{align*}
K_{02} &\leq |\doubleint_{\D(u.\ub)} \tr\chib_0\cdot \Psi(D_3 R)^{(s_1)}_g \cdot \Psi(D_3 R)^{(s_2)}_g| \lesssim \doubleint_{\D(u.\ub)} |\Psi(D_3 R)^{(s_1)}_g|^2 + |\Psi(D_3 R)^{(s_2)}_g|^2\\
& \leq  \sum (\int_0^{u} \|\Psi(D_3 R)_g\|_{L^2_{(sc)}({H}_{u'})}^2 + \delta^{-1} \int_0^{\ub} \|\Psi(D_3 R)_g\|_{L^2_{(sc)}({\Hb}_{\ub'})}^2).
\end{align*}
According to Gronwall's inequality, $K_{02}$ can be absorbed by the left hand side of the sum of \eqref{total energy for non anomalous L components} and \eqref{total energy for non anomalous Lb components}. Thus, we can ignore $I_{02}$ at this stage. Hence,
\begin{equation}\label{K_0_Lb}
K_{0} \lesssim C \delta^{\frac{1}{2}}.
\end{equation}

\subsubsection{Estimates for $K_1$}
One can argue exactly as in the beginning of Section \ref{estimates for K_1}, $\tr\chib_0$ does not appear as the connection coefficients in the integrand of $K_1$. We split $K_1$ into two terms
\begin{equation*}
K_1 = K_{11} + K_{12}
\end{equation*}
where the corresponding integrands for $K_{11}$ and $K_{12}$ are
\begin{equation*}
\psi^{(s_0)} \cdot \alpha(D_3 R) \cdot \Psi(D_3 R)^{(s_2)}_g, \quad \psi^{(s_0)} \cdot \alphab(D_3 R) \cdot \Psi(D_3 R)^{(s_2)}_g,
\end{equation*}
respectively.

We first control $K_{11}$. Up to an error of size $C\delta^{\frac{1}{4}}$, using Bianchi equations, we have
\begin{align*}
K_{11} &=  \doubleint_{\D(u,\ub)}\psi \cdot \alpha \cdot (\nabla_3 \Psi) ^{(s_2)} =  \doubleint_{\D(u,\ub)}\psi \cdot \alpha \cdot \nabla \Psi_g \\
&=  \doubleint_{\D(u,\ub)} \nabla \psi \cdot \alpha \cdot \Psi_g + \doubleint_{\D(u,\ub)} \psi \cdot \nabla \alpha \cdot \Psi_g.
\end{align*}

Now it is routine to bound those two integrals by $C\delta^{\frac{1}{4}}$. Hence,
\begin{equation*}
K_{11} \lesssim C \delta^{\frac{1}{4}}.
\end{equation*}

We now control $K_{12}$, we need to further decompose it into
\begin{equation*}
K_{12} = K_{121} + K_{122}= \doubleint_{\D(u,\ub)} \chih \cdot \alphab(D_3 R) \cdot \Psi(D_3 R)_g + \doubleint_{\D(u,\ub)}\psi \cdot \alphab(D_3 R) \cdot \Psi(D_3 R)_g
\end{equation*}
where $\psi \neq \chih$. We remark that $\chibh$ does not appear, otherwise $\Psi(D_3 R)_g = \alpha(D_3 R)$ which leads to a double anomaly.

To estimate $I_{121}$, observe that the appearance of $\chih$ is through the term
\begin{equation*}
{}^{(L)}\pi_{ab}\cdot Q(D_3 R)^{ab}{}_{N_1 N_2}.
\end{equation*}
Since $\alphab(D_3 R)$ also appears, this forces $N_1 = N_2 = \Lb$. Direct computations show $\Psi(D_3 R)_g = \rho(D_3 R)$ or $\Psi(D_3 R)_g = \sigma(D_3 R)$. Without loss of generality, we may assume
\begin{equation*}
K_{121} = \doubleint_{\D(u,\ub)} \chih \cdot \nabla_3 \alphab \cdot \nabla_3 \rho.
\end{equation*}
We use \eqref{NBE_Lb_rho} to replace $\nabla_3 \rho$. Up to an error of size $C \delta^{\frac{1}{4}}$, we have
\begin{align*}
K_{121} &= \doubleint_{\D(u,\ub)} \chih \cdot \nabla_3\alphab \cdot \nabla \betab = I_{1211}+I_{1212}+I_{1213}\\
&= \doubleint_{\D(u,\ub)} \nabla_3\chih \cdot \alphab \cdot \nabla \betab+\doubleint_{\D(u,\ub)} \chih \cdot \alphab \cdot \nabla_3\nabla \betab+ \int_{H_{u'}} \chih \cdot \alphab \cdot \nabla \betab \mid_{u'=0}^{u'=u}
\end{align*}
For the boundary integral $I_{1213}$,
\begin{equation*}
 K_{1213} \lesssim \sup_{u}|\int_{H_{u}} \chih \cdot \alphab \cdot \nabla \betab| \lesssim  \delta^{\frac{1}{2}}\|\chih\|_{L^{\infty}_{(sc)}}\|\alphab\|_{L^2_{(sc)}({H}_{u})}\|\nabla \betab\|_{L^2_{(sc)}({H}_{u})} \lesssim C \delta^{\frac{1}{2}}.
\end{equation*}
For $K_{1212}$, we compute commutator
\begin{equation*}
 [\nabla_3,\nabla ] \betab = -\chib \cdot \nabla \betab + {}^*\!\betab \cdot \betab + \frac{1}{2}(\eta +\etab)\nabla_3 \betab + \chib \cdot \eta \cdot \betab.
\end{equation*}
All terms are quadratic except the first one which is essentially linear in $\nabla \betab$. Thus, it is routine to derive
\begin{equation*}
|\doubleint_{\D(u,\ub)} \chih \cdot \alphab \cdot [\nabla_3, \nabla] \betab| \lesssim |\doubleint_{\D(u,\ub)} \chih \cdot \alphab \cdot \nabla \betab|+C \delta^{\frac{1}{4}} \lesssim C \delta^{\frac{1}{4}}.
\end{equation*}
Hence, up to an error of size $C\delta^{\frac{1}{4}}$, we have
\begin{equation*}
K_{1212} = \doubleint_{\D(u,\ub)} \chih \cdot \alphab \cdot \nabla \nabla_3 \betab = \doubleint_{\D(u,\ub)} \nabla \chih \cdot \alphab \cdot \nabla_3 \betab + \doubleint_{\D(u,\ub)} \chih \cdot \nabla \alphab \cdot \nabla_3 \betab.
\end{equation*}
Hence, we eliminate all the anomalies. Thus, it is routine to bound both integrals by $C \delta^{\frac{1}{4}}$ as well as
\begin{equation*}
K_{1212} \lesssim C \delta^{\frac{1}{4}}.
\end{equation*}
For $K_{1211}$, we use \eqref{NSE_Lb_chih} to replace $\nabla_4 \chih$, thus,
\begin{align*}
K_{1211} &=\doubleint_{\D(u,\ub)} ( -\frac{1}{2} \tr \chib \chih + \nabla \tensor \eta +2\omegab \chih -\frac{1}{2}\tr \chi \chibh +\eta \tensor \eta.) \cdot \alphab \cdot \nabla \betab\\
&= -\frac{1}{2}\tr \chib_0 \doubleint_{\D(u,\ub)} \chih  \cdot \alphab \cdot \nabla \betab + \doubleint_{\D(u,\ub)}\nabla \tensor \eta \cdot \alphab \cdot \nabla \betab + C \delta^{\frac{1}{4}}.
\end{align*}
For the first integral, we simply use $L^{\infty}_{(sc)}$ on $\chih$ to overcome the anomaly; for the second one, we control it as follows,
\begin{equation*}
\doubleint_{\D(u,\ub)}\nabla \tensor \eta \cdot \alphab \cdot \nabla \betab  \leq \delta^{\frac{1}{2}} \int_0^{u} \|\nabla \betab\|_{L^{2}_{(sc)}({H}_{u'})}\|\alphab\|_{L^4_{(sc)}({H}_{u'})}\|\nabla \tensor \eta\|_{L^4_{(sc)}({H}_{u'})}d u'\lesssim C \delta^{\frac{1}{2}}.
\end{equation*}
This yields the control of $I_{1211}$. Combined with the estimates for $K_{1212}$ and $K_{1213}$, we have
\begin{equation*}
K_{121} \lesssim C \delta^{\frac{1}{4}}.
\end{equation*}

To estimate $K_{122}$, we observe that in its integrand $\psi \cdot \alphab(D_3 R) \cdot \Psi(D_3 R)_g$, we have $\psi \neq \omegab$. Otherwise the total signature is less than $2$. This allow $\nabla_3 \psi$ to have $L^4_{(sc)}$ estimates (while $\nabla_3 \omegab$ does not, see \cite{K-R-09}). We can proceed exactly (much easier in reality) as we did for $K_{121}$ to show that $K_{12} \lesssim C \delta^{\frac{1}{4}}$. Combined with the estimates for $K_{11}$, we derive
\begin{equation*}
K_{1} \lesssim C \delta^{\frac{1}{4}}.
\end{equation*}

\subsubsection{Vanishing for $K_2$}
We claim $K_2 = 0$. The integrand of $K_2$ may have the following three types of terms
\begin{equation*}
(\psi + \tr\chib_0) \cdot \alpha(D_3 R) \cdot \alpha(D_3 R), \quad (\psi + \tr\chib_0) \cdot \alphab(D_3 R)\cdot \alphab(D_3 R), \quad (\psi + \tr\chib_0) \cdot \alpha(D_3 R) \cdot \alphab(D_3 R).
\end{equation*}
Since the total signature is either 2 or 3, this rules out the first two cases immediately. To rule out the third one, notice that its appearance is originally through the following term
\begin{equation*}
{}^{(N_0)}\!\pi_{\mu\nu}Q[D_3 R]^{\mu\nu}{}_{N_1 N_2}
\end{equation*}
Here, $N_1$ and $N_2$ are null. Once again,direct computations show that $\alpha \cdot \alphab$ never appears in $Q_{\mu\nu N_1 N_2}$.

Finally, we conclude that
\begin{equation}\label{K Lb}
K \lesssim C \delta^{\frac{1}{4}}.
\end{equation}

\subsection{Estimates for $J$}
The integrand of $J$ can be written schematically as
\begin{equation*}
(\psi + \tr\chib_0)^{(s_0)} \cdot \Psi(D_3 R)^{(s_1)} \cdot \Psi(D_\nu R)^{(s_2)}
\end{equation*}
with total signature $s_0 +s_1 +s_2 = 2$ or $3$. We split $J$ into
\begin{equation*}
J = J_0 + J_1 + J_2
\end{equation*}
where $J_k$ denotes the collection of terms with exactly $k$ anomalous curvature components.

\subsubsection{Estimates for $J_0$}
The integrand of $J_0$ must be one of the following forms
\begin{equation*}
\psi^{(s_0)} \cdot \Psi(D_3 R)^{(s_1)}_g \cdot \Psi(D_\nu R)^{(s_2)}_g, \quad \tr\chib_0\cdot \Psi(D_3 R)^{(s_1)}_g \cdot \Psi(D_\nu R)^{(s_2)}_g.
\end{equation*}
Thus,  those terms can be dealt with exactly as we have done for $K_0$ is last subsection (of the current section). Thus, modulo the terms which will be removed eventually by Gronwall's inequality, we derive
\begin{equation}\label{J_0_Lb}
J_{0} \lesssim C \delta^{\frac{1}{2}}.
\end{equation}

\subsubsection{Estimates for $J_1$}
According to the position of the anomaly, we split $J_1$ into
\begin{equation*}
J_1 = J_{11} + J_{12} + J_{13}
\end{equation*}
where the corresponding integrands for $J_{11}$, $J_{12}$ and $J_{13}$ are
\begin{equation*}
(\psi+\tr\chib_0) \cdot \alpha(D_3 R) \cdot \Psi(D_\nu R)_g, (\psi+\tr\chib_0) \cdot \alphab(D_3 R) \cdot \Psi(D_\nu R)_g, (\psi+\tr\chib_0) \cdot \Psi(D_3 R)_g \cdot \Psi(D_\nu R)_b
\end{equation*}
respectively. We use subindex $b$ to denote the anomaly ('b' for bad).

To estimate $J_{11}$, since the total signature is at most $3$, thus the signature of $\Psi(D_\nu R)_g$ is at most $1$. Using Bianchi equations, up to an error of size $C\delta^{\frac{1}{4}}$, we can replace the curvature term $\Psi(D_\nu R)^{(s_2)}_g$ by $\nabla \Psi_g$ and we aslo replace the mild $\alpha(D_3 R)$ anomaly by $\alpha$. Now we only consider the worst scenario. It should contain $\tr\chib_0$. In this case, according to signature considerations, $\nabla \Psi_g\neq \nabla \alphab$.  Hence we control the worst scenario for $J_{11}$ as follows,
\begin{align*}
J_{11} &\leq |\doubleint_{\D(u.\ub)} \tr\chib_0 \cdot \alpha \cdot \nabla \Psi_g| = |\doubleint_{\D(u.\ub)} \tr\chib_0 \cdot \nabla \alpha \cdot \Psi_g| \lesssim  \int_0^{u} \|\nabla \alpha\|_{L^2_{(sc)}({H}_{u})} \|\Psi_g\|_{L^2_{(sc)}({H}_{u})}\\
&\lesssim \Rone (\Izero + c(\Izero)\R^{\frac{1}{2}}+C\delta^{\frac{1}{8}})\lesssim c(\Izero)(\R+\Rb)^{\frac{3}{2}}+C\delta^{\frac{1}{8}}.
\end{align*}

To estimate $J_{12}$, recall that the integrand is originally from (modulo the Hodge dual part)
\begin{equation*}
 D_3 R_{X}{}^{\rho}{}_{Y}{}^{\delta} \cdot D^{\mu} \Lb^{\nu} \cdot D_\nu R_{\mu \rho Z \delta},
\end{equation*}
It has to be
\begin{equation*}
 D_3 R_{3}{}^{a}{}_{3}{}^{b} \cdot D^{\mu} \Lb^{\nu} \cdot D_\nu R_{\mu a 4 b},
\end{equation*}
If $(\mu, \nu) \notin \{(3,3), (3,c), (c,3), (c,d) | c,d=1,2\}$, we know  $D^{\mu} \Lb^{\nu} = 0$. We also notice $(\mu, \nu) \neq (c,d)$, otherwise $J_1$ contains a double anomaly in curvature components. Hence, neither $\chibh$ or $\tr\chib_0$ appears as connection coefficients. Furthermore, $\omega$ is also absent. Thus, $J_{12}$ is essentially $K_{122}$ in last subsection (of the current section). Hence,
\begin{equation*}
 J_{12} \lesssim C\delta^{\frac{1}{4}}.
\end{equation*}

To estimate $J_{13}$, we need to carefully analyze the structure of
\begin{equation*}
 D_3 R_{X}{}^{\rho}{}_{Y}{}^{\delta} \cdot D^{\mu} \Lb^{\nu} \cdot D_\nu R_{\mu \rho Z \delta},
\end{equation*}
As we pointed out at the beginning of the section, the possible anomalies for last term are $\alpha(D_4 R)$, $\alphab(D_3 R)$, $\alpha(D_3 R)$ and $\beta(D_c R)$. First of all, $\alpha(D_4 R)$ can not occur since $D^\mu \Lb^4 =0$. Thus, we have three cases left to deal with.

If $\Psi(D_\nu R) = \alphab(D_3 R)$, thus $\nu =3$ and the connection coefficient $D^\mu \Lb^3 \neq \chibh ~\text{or}~ \tr\chib_0$. The integrand looks like
\begin{equation*}
 \nabla \Psi_g \cdot \psi_g \cdot \nabla_3 \alphab.
\end{equation*}
Those terms were estimated in $K_{122}$ in the last subsection.

If $\Psi(D_\nu R) = \alpha(D_3 R)$ or $\beta(D_c R)$, those anomalies are mild hence can be replaced by $\alpha$. Combined with Bianchi equations if necessary, up to an error of size $C\delta^{\frac{1}{4}}$, the corresponding integrand looks like
\begin{equation*}
(\psi + \tr\chib_0) \cdot \nabla \Psi_g \cdot \alpha.
\end{equation*}
Those terms have been treated in $J_{11}$. Thus, we can derive estimates for $J_{13}$. Combined with $J_{11}$ and $J_{12}$, we have
\begin{equation}\label{J_1_Lb}
J_{1} \lesssim c(\Izero)(\R+\Rb)^{\frac{3}{2}}+C\delta^{\frac{1}{8}}.
\end{equation}

\subsubsection{Estimates for $J_2$}
In this case, both the curvature components in the integrand
\begin{equation*}
 D_3 R_{X}{}^{\rho}{}_{Y}{}^{\delta} \cdot D^{\mu} \Lb^{\nu} \cdot D_\nu R_{\mu \rho Z \delta},
\end{equation*}
are anomalous. According to the different anomalies of $ D_3 R_{X}{}^{\rho}{}_{Y}{}^{\delta}$, we decompose $J_{2}$ into
\begin{equation*}
 J_2 =J_{21}+J_{22}.
\end{equation*}
For $J_{21}$, $D_3 R_{X}{}^{\rho}{}_{Y}{}^{\delta} = D_3 R_{4a4b}$; for $J_{22}$, $D_3 R_{X}{}^{\rho}{}_{Y}{}^{\delta} = D_3 R_{3a3b}$.

To estimate $J_{21}$, we observe that $D_\nu R_{\mu \rho Z \delta} = D_3 R_{3 a 3 b}$ because the total signature is at most $3$. Thus, up to an error of size $C\delta^{\frac{1}{4}}$,
\begin{align*}
J_{21} &= \doubleint_{\D(u,\ub)} \omega \cdot \nabla_4 \alpha \cdot \nabla_3 \alphab = J_{211}+J_{212}+J_{213}\\
&= \doubleint_{\D(u,\ub)} \nabla_3 \omega \cdot \nabla_4 \alpha \cdot \alphab + \int_{H_{u'}} \omega \cdot \alphab \cdot \nabla_4 \alpha \mid_{u'=0}^{u'=u} +\doubleint_{\D(u,\ub)} \omega \cdot \alphab \cdot \nabla_3 \nabla_4 \alpha.
\end{align*}
For $J_{211}$, according to \eqref{NSE_Lb_omega}, up to an error term of size $C\delta^{\frac{1}{4}}$, we can replace $\nabla_3 \omega$ by $\rho$, thus,
\begin{align*}
J_{211} &\leq \delta^{\frac{1}{2}} \int_{0}^{u} \|\rho\|_{L^4_{(sc)}({H}_{u})} \|\alphab\|_{L^4_{(sc)}({H}_{u})} \|\nabla_4 \alpha\|_{L^2_{(sc)}({H}_{u})}
\end{align*}
We can then repeat the proof of Lemma \ref{L4 alpha precise} to derive
\begin{equation*}
 \|(\rho, \alphab)\|_{L^4_{(sc)}(H_u)} \lesssim c(\Izero) + c(\Izero)\R^{\frac{1}{2}} + C\delta^{\frac{1}{4}}.
\end{equation*}
Thus,
\begin{align*}
J_{211} &\leq \delta^{\frac{1}{2}} \int_{0}^{u} (\R + C\delta^{\frac{1}{4}})^2 (\delta^{-\frac{1}{2}}\Izero + C\delta^{\frac{1}{4}}) \lesssim C\delta^{\frac{1}{4}} + c(\Izero)\int_{0}^{u} \R^2 du
\end{align*}
The last integral will be eventually removed by Gronwall's inequality, thus we could write
\begin{align*}
J_{211}  \lesssim C\delta^{\frac{1}{4}}.
\end{align*}
For $J_{212}$, in view of Proposition \ref{pricise on chi chib}, we have
\begin{equation}\label{precise estimates on omega}
  \|\omega\|_{L^4_{(sc)}({H}_{u})} \lesssim \Rzero^{\frac{1}{2}} \R^{\frac{1}{2}} + C \delta^{\frac{1}{4}} \lesssim  \R^{\frac{3}{4}} + C \delta^{\frac{1}{4}}.
\end{equation}
We derive
\begin{align*}
 J_{212} &\lesssim |\int_{H_{u}} \omega \cdot \alphab \cdot \nabla_4 \alpha| \leq \delta^{\frac{1}{2}} \|\omega\|_{L^4_{(sc)}({H}_{u})} \|\alphab\|_{L^4_{(sc)}({H}_{u})} \|\nabla_4 \alpha\|_{L^2_{(sc)}({H}_{u})}\\
&\lesssim \delta^{\frac{1}{2}} ( \R^{\frac{3}{4}} + C \delta^{\frac{1}{4}})(\R+ C\delta^{\frac{1}{4}})(\delta^{-\frac{1}{2}}c(\Izero) +  C \delta^{\frac{1}{4}}) \lesssim  c(\Izero)\R^{\frac{7}{4}} + C \delta^{\frac{1}{4}}.
\end{align*}
For $J_{213}$, we switch the order of $\nabla_3$ and $\nabla_4$ acting on $\alpha$,
\begin{align*}
J_{213} &= \doubleint_{\D(u,\ub)} \omega \cdot \alphab \cdot \nabla_4 \nabla_3 \alpha + \doubleint_{\D(u,\ub)} \omega \cdot \alphab \cdot [\nabla_3,\nabla_4] \alpha
\end{align*}
We ignore the commutator term since it is bounded by $C\delta^{\frac{1}{4}}$. The proof is routine. Thus, up to an error of size $C\delta^{\frac{1}{4}}$,
\begin{equation*}
 J_{213}=\doubleint_{\D(u,\ub)} \nabla_4 \omega \cdot  \alphab \cdot \nabla_3 \alpha + \doubleint_{\D(u,\ub)}  \omega \cdot  \nabla_4 \alphab \cdot \nabla_3 \alpha+ \int_{\Hb_{\ub'}} \omega \cdot \alphab \cdot \nabla_3 \alpha \mid_{\ub=0}^{\ub'=\ub}
\end{equation*}
The last two integrals are bounded by $C \delta^{\frac{1}{4}}$. In fact, we can replace the mild anomaly $\nabla_3 \alpha$ by $\alpha$. Then the proof becomes routine. We now move on to the most dangerous term in $J_{213}$,
\begin{equation*}
 \doubleint_{\D(u,\ub)} \nabla_4 \omega \cdot  \alphab \cdot \nabla_3 \alpha
\end{equation*}
It contains $\nabla_4 \omega$ which does not have $L^{4}_{(sc)}$ estimates. We can only rely on its $L^{2}_{(sc)}$ estimates. It requires an extra integration by parts. Up to an error of size $C\delta^{\frac{1}{4}}$, we use \eqref{NBE_Lb_alpha} to derive
\begin{equation*}
 \doubleint_{\D(u,\ub)} \nabla_4 \omega \cdot  \alphab \cdot \nabla_3 \alpha =\doubleint_{\D(u,\ub)} \nabla_4 \omega \cdot  \alphab \cdot \nabla \beta + \doubleint_{\D(u,\ub)} \nabla_4 \omega \cdot  \alphab \cdot \alpha
\end{equation*}
The second integral is bounded by $C\delta^{\frac{1}{4}}$, because we can put $L^2_{(sc)}$ on $\nabla_4 \omega$ and $L^4_{(sc)}$ on curvature components. We move around the derivatives in the first integral,
\begin{align*}
 \doubleint_{\D(u,\ub)} \nabla_4 \omega \cdot \alphab \cdot \nabla \beta &= \doubleint_{\D(u,\ub)} \omega \cdot \nabla_4 \alphab \cdot \nabla \beta +  \doubleint_{\D(u,\ub)}\omega \cdot \alphab \cdot \nabla_4 \nabla \beta+ \int_{\Hb_{\ub'}} \omega \cdot \alphab \cdot \nabla \beta \mid_{\ub=0}^{\ub'=\ub}\\
&= \doubleint_{\D(u,\ub)}\omega \cdot \alphab \cdot \nabla_4 \nabla \beta + C\delta^{\frac{1}{4}}\\
&= \doubleint_{\D(u,\ub)}\omega \cdot \alphab \cdot [\nabla_4 ,\nabla ]\beta + \doubleint_{\D(u,\ub)}\omega \cdot \alphab \cdot  \nabla  \nabla_4 \beta\\
&= \doubleint_{\D(u,\ub)}\omega \cdot \alphab \cdot [\nabla_4 ,\nabla ]\beta + \doubleint_{\D(u,\ub)} \nabla \omega \cdot \alphab \cdot   \nabla_4 \beta + \doubleint_{\D(u,\ub)}\omega \cdot \nabla \alphab \cdot  \nabla_4 \beta
\end{align*}
Thus, we removed all the anomalies. This yields the bound for $J_{213}$. Together with $J_{211}$ and $J_{212}$, we derive
\begin{equation*}
 J_{21} \lesssim  c(\Izero)\R^{\frac{7}{4}} + C \delta^{\frac{1}{4}}.
\end{equation*}

We now estimate $J_{22}$. Its integrand is in the following form
\begin{equation*}
 D_3 R_{3}{}^{a}{}_{3}{}^{b} \cdot D^{\mu} \Lb^{\nu} \cdot D_\nu R_{\mu a 4 b},
\end{equation*}
Since $D^{4} \Lb^{\nu} =D^{\mu} \Lb^{4} = 0$ and $ D_\nu R_{\mu a 4 b}$ is an anomaly, the integrand of $J_{22}$ is reduced to the following form
\begin{equation*}
 D_3 R_{3}{}^{a}{}_{3}{}^{b} \cdot D^{d} \Lb^{c} \cdot D_c R_{d a 4 b} = D_3 R_{3}{}^{a}{}_{3}{}^{b} \cdot \chib_{dc} \cdot D_c R_{d a 4 b}
\end{equation*}
We further decompose $J_{22}$ into
\begin{equation*}
 J_{22}=J_{221}+J_{222}
\end{equation*}
such that the corresponding integrands of $J_{221}$ and $J_{222}$ are
\begin{equation*}
 D_3 R_{3}{}^{a}{}_{3}{}^{b} \cdot \chibh_{cd} \cdot D_c R_{d a 4 b}, \quad D_3 R_{3}{}^{a}{}_{3}{}^{b} \cdot \tr\chib_{0} \cdot D_c R_{c a 4 b},
\end{equation*}
respectively.

For $J_{221}$, up to an error term of size $C\delta^{\frac{1}{4}}$, we replace $D_c R_{d a 4 b}$ by $\nabla \beta + \alpha$ to derive
\begin{equation*}
 J_{221} = \doubleint_{\D(u,\ub)} \chibh \cdot \nabla_3 \alphab \cdot \nabla \beta + \doubleint_{\D(u,\ub)} \chibh \cdot \nabla_3 \alphab \cdot \alpha = J_{2211} + J_{2212}.
\end{equation*}
For $J_{2211}$,
\begin{equation*}
 J_{2211}=\doubleint_{\D(u,\ub)}  \nabla_3 \chibh \cdot\alphab \cdot \nabla \beta + \doubleint_{\D(u,\ub)} \chibh \cdot\alphab \cdot \nabla_3 \nabla \beta+ \int_{H_{u'}} \chibh \cdot\alphab \cdot \nabla \beta \mid_{\ub=0}^{\ub'=\ub}
\end{equation*}
The last two integrals are bounded by $C\delta^{\frac{1}{4}}$. The proof is routine. For the first term, we use \eqref{NSE_Lb_chib} to replace $\nabla_3 \chibh$, thus
\begin{equation*}
 J_{2211} = \doubleint_{\D(u,\ub)}  (-\tr \chib \, \chibh -2\omegab \chibh -\alphab) \cdot\alphab \cdot \nabla \beta + C\delta^{\frac{1}{4}}\lesssim C \delta^{\frac{1}{4}}.
\end{equation*}
For $J_{2212}$, we have
\begin{align*}
 J_{2212} &= \doubleint_{\D(u,\ub)} \nabla_3 \chibh \cdot \alphab \cdot \alpha +  \doubleint_{\D(u,\ub)}  \chibh \cdot \alphab \cdot \nabla_3 \alpha+ \int_{H_{u'}} \chibh \cdot \alphab \cdot \alpha \mid_{u=0}^{u'=u}
\end{align*}
Combining with\eqref{NSE_Lb_chib} and \eqref{anomalous behavior for alpha(D_3 R)}, each of the above integrals generates the most difficult term
\begin{equation*}
 \int_{H_{u}} \chibh \cdot \alphab \cdot \alpha.
\end{equation*}
In view of Proposition \ref{pricise on chi chib},
\begin{equation*}
 \|\chibh\|_{L^4_{(sc)}(H)} \lesssim c(\Izero) \delta^{-\frac{1}{4}} + \Rzerob^{\frac{1}{2}}\Roneb^{\frac{1}{2}} + C \delta^{\frac{1}{4}} \lesssim c(\Izero) \delta^{-\frac{1}{4}} + c(\Izero)\R^{\frac{3}{4}} + C \delta^{\frac{1}{4}}.
\end{equation*}
Thus, in view of Lemma \ref{L4 alpha precise} we have
\begin{align*}
 |\int_{H_{u}} \chibh \cdot \alphab \cdot \alpha| &\lesssim \delta^{\frac{1}{2}}(c(\Izero) \delta^{-\frac{1}{4}} +c(\Izero) \R^{\frac{3}{4}} + C \delta^{\frac{1}{4}})(\Izero  + c(\Izero)\R^{\frac{1}{2}} + C \delta^{\frac{1}{8}})\delta^{-\frac{1}{4}}(\Izero  + \R^{\frac{1}{2}} + C \delta^{\frac{1}{4}})\\
&\lesssim c(\Izero)  + c(\Izero) \R^{\frac{7}{4}} + C \delta^{\frac{1}{16}}.
\end{align*}
Thus, yields the estimates on $J_{2212}$. Combined with $J_{2211}$, we have
\begin{align*}
 J_{221} \lesssim c(\Izero)  + c(\Izero) \R^{\frac{7}{4}} + C \delta^{\frac{1}{16}}.
\end{align*}

The last mission we need to complete is to show that $J_{222}$ vanishes. In fact, the expression of the integrand
\begin{equation*}
 D_3 R_{3}{}^{a}{}_{3}{}^{b} \cdot D_c R_{c a 4 b} = \alphab(D_3 R)_{ab} \varepsilon_{ca}\!^*\beta(D_c R)_b
\end{equation*}
is not sufficient to show $J_{222}$ vanishes. And we can not control it. In view of the formula \eqref{divergence of Q}, we have to also investigate the part from the Hodge dual to see the cancelation. We now show that
\begin{equation*}
 D_3 \!^*R_{3}{}^{a}{}_{3}{}^{b}  \cdot D_c \!^*R_{c a 4 b} =\alphab(D_3 \!^*R)_{ab} \varepsilon_{ca}\!^*\beta(D_c \!^* R)_b=- \alphab(D_3 R)_{ab} \varepsilon_{ca}\!^*\beta(D_c R)_b.
\end{equation*}

We recall some basic formulas from \cite{Ch-K}. For 1-form $\beta$, $\!^* \beta_a = \varepsilon_{ab}\beta^b$. For a symmetric traceless 2-tensor $\alpha$, $\!^* \alpha_{ab} = \varepsilon_{ac}\alpha^{c}{}_{b}$ and $\alpha^*_{ab} = \alpha_{a}{}^{c}\varepsilon_{cb}$. One can easily check that $\!^* \alpha = - \alpha^*$.

We also have $\alpha(\!^* W) = \!^*\alpha(W) = - \alpha(W)^*$, $\beta(\!^* W) =  - \beta(W)^*$ and $W_{ab4c}$. Thus,
\begin{align*}
 \alphab(D_3 \!^*R)_{ab} \varepsilon_{ca}\!^*\beta(D_c \!^* R)_b&=[-\alphab(D_3 R)^*_{ab}] \varepsilon_{ca}\varepsilon_{bd}\beta(D_c \!^* R)_d=\alphab(D_3 R)^*_{ab} \varepsilon_{ca}\varepsilon_{bd}\!^*\beta(D_c  R)_d\\
&=	\alphab(D_3 R)^*_{ae} \varepsilon_{eb}\varepsilon_{ca}\varepsilon_{bd}\!^*\beta(D_c  R)_d = -\alphab(D_3 R)^*_{ae} \varepsilon_{ca}\delta_{ed}\!^*\beta(D_c  R)_d\\
&=- \alphab(D_3 R)_{ab} \varepsilon_{ca}\!^*\beta(D_c R)_b.
\end{align*}
Thus, $J_{222} = 0$. 
\begin{remark}
This cancelation can be explained in another way. To derive energy estimates at this situation, i.e. $N=\Lb$ and $(X,Y,Z) = (L,\Lb,Lb)$, by using Bel-Robinson tensors, it is equivalent to first commute \eqref{NBE_Lb_rho},\eqref{NBE_Lb_sigma} and \eqref{NBE_L_betab} with $\nabla_3$, and then multiply the top order terms and integrate by parts directly on $\D(u,ub)$. In this way, we see directly that $J_{222}$ type term do not appear.
\end{remark}

Putting all the estimates together, we derive
\begin{align*}
 J_{2} \lesssim c(\Izero)  + c(\Izero) \R^{\frac{7}{4}} + C \delta^{\frac{1}{16}}.
\end{align*}
Combined with the estimates for $J_{0}$ and $J_1$, we have
\begin{equation*}
 J \lesssim c(\Izero)  + c(\Izero) (\R+\Rb)^{\frac{7}{4}} + C \delta^{\frac{1}{16}}.
\end{equation*}
Combined again with \eqref{total energy for non anomalous Lb components}, \eqref{Estimate for Lb I} and \eqref{K Lb}, we derive
\begin{equation}\label{Energy Estimates part B}
\|(\nabla_3 \beta, \nabla_3 \rho, \nabla_3 \sigma)\|_{L^2_{(sc)}(H_u)}^2 + \|(\nabla_3\rho, \nabla_3\sigma, \nabla_3\betab) \|_{L^2_{(sc)}({\Hb}_{\ub})}^2 \lesssim c(\Izero)  + c(\Izero) (\R+\Rb)^{\frac{7}{4}} + C \delta^{\frac{1}{16}}.
\end{equation}
Considering $\|\nabla_3\betab \|_{L^2_{(sc)}({\Hb}_{\ub})}$ in \eqref{Energy Estimates part B}, combined \eqref{NBE_Lb_betab} and standard elliptic estimates for Hodge systems, we have
\begin{equation*} \label{nabla alphab on Hb}
\|\nabla \alphab \|_{L^2_{(sc)}({\Hb}_{\ub})} \lesssim c(\Izero)  + c(\Izero) (\R+\Rb)^{\frac{7}{8}} + C \delta^{\frac{1}{32}}.
\end{equation*}
Considering $\nabla_3 \rho$ and $\nabla_3 \sigma$ in \eqref{Energy Estimates part B}, combined \eqref{NBE_Lb_rho}, \eqref{NBE_Lb_sigma} and elliptic estimates, we derive
\begin{equation*}
\|\nabla \betab \|_{L^2_{(sc)}({H}_{u})} +\|\nabla \betab \|_{L^2_{(sc)}({\Hb}_{\ub})} \lesssim c(\Izero)  + c(\Izero) (\R+\Rb)^{\frac{7}{8}} + C \delta^{\frac{1}{32}}.
\end{equation*}
Hence,
\begin{equation}\label{Close Bootstrap 2}
\|\nabla \betab\|_{L^2_{(sc)}({H}_{u}^{(0,u)})} + \|(\nabla \betab, \nabla \alphab)\|_{L^2_{(sc)}({\Hb}_{\ub}^{(0,u)})} \lesssim c(\Izero)  + c(\Izero) (\R+\Rb)^{\frac{7}{8}} + C \delta^{\frac{1}{32}}.
\end{equation}

We combine \eqref{estimates for nabla_4 alpha on H}, \eqref{estimates for nabla_3 alphab on Hb}, \eqref{Close Bootstrap 1} and \eqref{Close Bootstrap 2} to derive that
\begin{equation*}
 \R + \Rb \lesssim c(\Izero) (1+ \R^{\frac{7}{8}}) +C\delta^{\frac{1}{32}}.
\end{equation*}
In view of Section \ref{structure of the proof}, this completes the proof of the \textbf{Main Estimates} of the paper.




\end{document}